\documentclass[10pt,twoside]{amsart}
\usepackage{amsmath}
\usepackage{amsthm}
\usepackage{amsfonts}
\usepackage{amssymb}
\usepackage{latexsym}
\usepackage{mathrsfs}
\usepackage{amsmath}
\usepackage{amsthm}
\usepackage{amsfonts}
\usepackage{amssymb}
\usepackage{latexsym}
\usepackage{geometry}
\usepackage{dsfont}
\usepackage{enumitem}
\usepackage[dvips]{graphicx}
\usepackage{color}
\usepackage[all]{xy}
\usepackage[pagebackref]{hyperref}
\renewcommand*{\backref}[1]{}\renewcommand*{\backrefalt}[4]%{\ifcase #1 (\tt not cited)\or (\tt cited on page~#2)}

\theoremstyle{plain}
\newtheorem{theorem}{Theorem}[section]
\newtheorem{proposition}[theorem]{Proposition}
\newtheorem{lemma}[theorem]{Lemma}
\newtheorem{corollary}[theorem]{Corollary}
\theoremstyle{definition}
\newtheorem{example}[theorem]{Example}
\newtheorem{definition}[theorem]{Definition}

\theoremstyle{definition}

\theoremstyle{remark}
\newtheorem{remark}[theorem]{Remark}

\def\fkm{{\mathfrak m}}

\def\Max{{\rm Max}}
\def\hom{\operatorname{hom}}
\def\Hom{\operatorname{Hom}}

\def\Ext{{\rm Ext}}
\def\Supp{{\rm Supp}}
\def\Ann{{\rm Ann}}
\def\ann{{\rm Ann}}
\def\Mod{{\rm Mod}}
\def\GV{{\rm GV}}
\def\tor{{\rm tor}}
\def\tr{{\rm tr}}

\def\Im{{\rm Im}}
\def\coker{{\rm Coker}}
\def\E{{\rm E}}
\def\Spec{{\rm Spec}}

\begin{document}

\title[Relative Faithful Exact Functors]{Relative Faithful Exact Functors and Their Applications to Homological Modules}

\author [Zhang] {Xiaolei Zhang}
\address{(Zhang) School of Mathematics and Statistics,	Tianshui Normal University, Tianshui 741001, China}
\email{zxlrghj@163.com}
\author[Qiao]{Lei Qiao}
\address{(Qiao)  School of Mathematical Sciences, Sichuan Normal University,
	Chengdu, 610066, China}
\email{lqiao@sicnu.edu.cn}
\author [Kim] {Hwankoo Kim$^{\dag}$}
\address{(Kim) Division of Computer \& Information Engineering, Hoseo University, Asan 31499, Republic of Korea}
\email{hkkim@hoseo.edu}


%\thanks{This research was supported by the National Natural Science Foundation of China (Grant No. 12101515).}

\thanks{$2020$ Mathematics Subject Classification: 18A22, 13C10, 13C11.}

\thanks{$\dag$ Corresponding Author}

\date{\today}

\begin{abstract}
The notions of faithfully projective, faithfully flat, and faithfully injective modules--defined as modules for which the three classical homological functors are both faithful and exact--play fundamental roles across various areas of algebra. In this paper, we extend these notions to the setting of $w$-operation theory. By introducing the concept of $w$-faithfully exact functors, we define and investigate the notions of $w$-faithfully projective, $w$-faithfully flat, and $w$-faithfully injective modules. We establish their fundamental properties and demonstrate their effectiveness in generalizing classical results.
\end{abstract}

\maketitle
			
\section{introduction}

Throughout this paper, all rings are commutative with identity and all modules are unitary.

The concept of a faithfully exact functor, introduced and systematically studied by Ishikawa in his seminal 1964 paper \cite{I64}, stands as a cornerstone of homological algebra. Ishikawa's work provided a unified framework for understanding various classes of modules--projective, flat, and injective--by characterizing them via the behavior of the associated $\Hom$ and tensor functors. A functor is said to be faithfully exact if it is exact and reflects exact sequences, a condition equivalent to being both exact and faithful (that is, it detects nonzero objects and nonzero morphisms).

Faithfully projective modules are projective modules $P$ over a ring $R$ for which the functor $\Hom_R(P,-)$ is faithful. They play a central role in Morita theory (see \cite{AF92}). Similarly, faithfully flat modules are flat modules $F$ over a ring $R$ such that the functor $F\otimes_R-$ is faithful. Faithfully flat descent has become indispensable in Gorenstein homological algebra, algebraic $K$-theory, and algebraic geometry (see \cite{CKL17,L25,M24}). 

Dually, faithfully injective modules are injective modules $E$ for which the contravariant functor $\Hom_R(-,E)$ is faithful. Such modules play a fundamental role in duality theories. For instance, over a Noetherian ring, Matlis duality establishes a contravariant equivalence between certain categories of modules via the $\Hom$ functor into a faithfully injective cogenerator (see \cite{EJ11}). For broader generalizations and further applications, we refer the reader to \cite{CH25}.

The $w$-operation, introduced by Wang and McCasland \cite{WM97}, represents a significant development in multiplicative ideal theory and its applications to module theory and homological algebra. As Zafrullah observed in \cite{Z00}, ``What is so nice about the $w$-operation is that it is smoother than the $t$-operation,'' and ``Smoothness of this sort can only mean one thing: that one can bring in much more homological algebra than with other star operations.'' 

The notions of generalized projective, flat, and injective modules within $w$-operation theory have been extensively introduced and studied by Kim and Wang. In 2014, they \cite{KM14-lcm} defined $w$-flat modules to be $R$-modules $F$ such that $w$-exact sequences are preserved under the tensor functor $-\otimes_R F$. Subsequently, in \cite{WK14}, they defined $w$-injective modules to be $R$-modules $E$ for which $w$-exact sequences are preserved under the functor $\Hom_R(-,L(E))$. In 2015, they \cite{WK15} introduced the notion of \textit{$w$-projective modules}, namely $R$-modules $P$ such that $\Ext^1_R(L(P),N)$ is $\GV$-torsion for all torsion-free $w$-modules $N$. 

A substantial body of literature shows that these homological notions in $w$-operation theory play central roles in the study of important classes of rings, such as Krull domains, P$v$MDs, $w$-Noetherian rings, and $w$-coherent rings (see \cite{BKW14,WK14,wq15,WQ19,fq19,WZ18,xw18,XW18,zxl20,zw21}).

Parallel to the development of faithfully exact functors, the theory of torsion theories and quotient categories has provided powerful tools for the study of module categories. The $w$-operation induces a closure operation on modules that is closely connected to the notion of $\GV$-ideals and to a hereditary torsion theory of finite type, denoted by $\tau_w=(\mathcal{T}(R),\mathcal{F}(R))$. Here, $\mathcal{T}(R)$ denotes the class of $\GV$-torsion modules and $\mathcal{F}(R)$ the class of $\GV$-torsion-free modules. 

This torsion theory gives rise to a reflexive full subcategory $(R,w)\text{-mod}$ consisting of all $w$-modules. This category is not only a Grothendieck category but also the heart of the quotient category 
\[
\mathsf{Q}_w(R)=\mathrm{Mod}(R)/\mathcal{T}(R),
\]
which is equivalent to $(R,w)\text{-mod}$. The reflector 
\[
L:\mathrm{Mod}(R)\longrightarrow (R,w)\text{-mod},
\quad 
L(M)=(M/\mathrm{tor}_{\GV}(M))_w,
\]
is exact and annihilates $\GV$-torsion modules.

The main motivation of this paper is to transplant Ishikawa's theory of faithfully exact functors from the category of all $R$-modules to the more refined setting of the $w$-module category. The central innovation is the introduction of the notion of a $w$-faithfully exact functor (Definition~\ref{def:w-faithful-functor}). A functor 
\[
G: \mathrm{Mod}(R) \longrightarrow \mathcal{A}
\]
into an abelian category $\mathcal{A}$ is said to be $w$-faithfully exact if it is exact, annihilates the class $\mathcal{T}(R)$ of $\GV$-torsion modules, and the induced functor 
\[
\overline{G}: \mathsf{Q}_w(R) \longrightarrow \mathcal{A}
\]
on the Serre quotient is faithfully exact. This definition captures the idea of exactness and faithfulness ``up to $\GV$-torsion,'' thereby making precise how such functors interact with the underlying $w$-structure.

The core contribution of the paper is the systematic development of the theory of $w$-faithfully exact functors and its application to the definition and characterization of three fundamental classes of modules within the $w$-framework:
\begin{itemize}
    \item \textbf{$w$-Faithfully Projective Modules (Section~4).}
    An $R$-module $P$ is defined to be $w$-faithfully projective if the functor 
    \[
    H_w(-)=\mathrm{hom}_w(P,-)=\mathrm{Hom}_R(L(P),L(-))
    \]
    is $w$-faithfully exact. We establish a comprehensive list of equivalent conditions, culminating in a $w$-analogue of Ishikawa's \cite[Theorem~1.1]{I64} (Theorem~\ref{thm:w-Ishi-22}). Among the key criteria are the nonvanishing of $\mathrm{hom}_w(P,R/\mathfrak m)$ for every maximal $w$-ideal $\mathfrak m$ and the local condition $P_{\mathfrak m}\neq 0$ for all $\mathfrak m\in w\text{-Max}(R)$. In particular, a finitely generated $w$-projective module $P$ is $w$-faithfully projective if and only if it satisfies these support conditions. Proposition~\ref{prop:w-Ishikawa24} provides a structural characterization: $P$ is $w$-faithfully projective if and only if $R$ is a $w$-direct summand of a finite direct sum of copies of $L(P)$. This yields a precise $w$-analogue of the classical characterization of a progenerator.

    \item \textbf{$w$-Faithfully Flat Modules (Section~5).}
    Using the $w$-tensor product $M\otimes_R^w N=L(M\otimes_R N)$, an $R$-module $M$ is defined to be $w$-faithfully flat if the functor 
    \[
    U_w(-)=(-)\otimes_R^w M
    \]
    is $w$-faithfully exact. Theorem~\ref{thm:w-def21-plus-full} establishes several equivalent characterizations paralleling the classical theory. Most notably, we obtain the following local criterion: a $w$-flat module $M$ is $w$-faithfully flat if and only if $M_{\mathfrak m}$ is faithfully flat over the local ring $R_{\mathfrak m}$ for every maximal $w$-ideal $\mathfrak m$. Thus, the global $w$-property reduces to a familiar local condition.

    \item \textbf{$w$-Faithfully Injective Modules (Section~6).}
    An $R$-module $E$ is called $w$-faithfully injective if the contravariant functor 
    \[
    V_w(-)=\mathrm{hom}_w(-,E)
    \]
    is $w$-faithfully exact. Theorem~\ref{thm:w-ishi-31} provides several equivalent characterizations, including criteria involving simple modules and residue fields. A particularly strong result (Corollary~\ref{cor:w-Ishikawa33}) shows that a $w$-injective module $E$ is $w$-faithfully injective if and only if its $w$-envelope $L(E)$ contains an isomorphic copy of every $\tau_w$-simple module. This is further equivalent to the condition that $L(E)$ has a direct summand isomorphic to the canonical module 
    \[
   Q=E_R\big(\bigoplus_{\mathfrak{m}} R/\mathfrak{m}\big).
    \]
    This connects the abstract functorial definition with concrete structural embeddings. Furthermore, Proposition~\ref{prop:w-Ishikawa36} establishes a perfect $w$-analogue of Ishikawa duality: an $R$-module $M$ is $w$-faithfully flat if and only if $\mathrm{hom}_w(M,E)$ is $w$-faithfully injective for every $w$-faithfully injective module $E$, and conversely.
\end{itemize}

\section{New Perspectives on $w$-Operation Theory}

Let $R$ be a commutative ring. A finitely generated ideal $J$ of $R$ is called a 
\emph{Glaz--Vasconcelos ideal} (briefly, a \emph{$\GV$-ideal}), denoted by 
$J\in \GV(R)$, if the canonical homomorphism
\[
\varphi_J:\ R\longrightarrow \Hom_R(J,R), 
\qquad 
r\longmapsto (j\mapsto rj),
\]
is an isomorphism.

For an $R$-module $M$, the \emph{$\GV$-torsion submodule} of $M$ is defined by
\[
\tor_{\GV}(M)
:=\{\,x\in M \mid \exists\,J\in \GV(R)\ \text{such that } Jx=0\,\}.
\]
Set
\[
F(M):=M/\tor_{\GV}(M).
\]
Then there is a short exact sequence of $R$-modules
\[
0\longrightarrow \tor_{\GV}(M)
\longrightarrow M
\longrightarrow F(M)
\longrightarrow 0.
\]

\begin{definition}\label{def:two-closures}
Let $R$ be a commutative ring.

\medskip
\noindent\textup{(1) (\cite[Definition 6.2.1]{WK24})}
If $M$ is $\GV$-torsion-free, we define its \emph{$w$-envelope} (or \emph{$w$-closure}) 
inside an injective hull by
\[
M_w
:=\big\{\,x\in \E(M)\ \bigm|\ \exists\,J\in \GV(R)\ 
\text{ such that } Jx\subseteq M\,\big\},
\]
where $\E(M)$ denotes the injective hull of $M$ as an $R$-module. 
This agrees with the standard definition of the $w$-envelope for $\GV$-torsion-free modules in the literature.

\medskip
\noindent\textup{(2)}
For an arbitrary $R$-module $M$, we define its \emph{extended $w$-envelope} by
\[
L(M)
:=\big\{\,x\in \E(F(M))\ \bigm|\ \exists\,J\in \GV(R)\ 
\text{ such that } Jx\subseteq F(M)\,\big\}
\subseteq \E(F(M)).
\]
Equivalently,
\[
L(M)=(F(M))_w.
\]
In particular, if $M$ is $\GV$-torsion-free, then $L(M)=M_w$.
\end{definition}

A $\GV$-torsion-free $R$-module $M$ is called a \emph{$w$-module} if $M_w=M$.
An ideal $I\lhd R$ is said to be a \emph{$w$-ideal} if it is a $w$-module as an $R$-module.
A ring $R$ is called a \emph{DW-ring} if every ideal of $R$ is a $w$-ideal; 
equivalently, $\GV(R)=\{R\}$.
The set of all maximal $w$-ideals, denoted by $w\text{-}\Max(R)$, 
is a subset of $\Spec(R)$.

It is straightforward to verify that a finitely generated ideal $J$ of $R$ 
is a $\GV$-ideal if and only if
\[
J\nsubseteq \mathfrak m 
\qquad \text{for every } \mathfrak m\in w\text{-}\Max(R).
\]

\begin{proposition}
\label{prop:envelope-equiv}
Let $M$ be a $\GV$-torsion-free $R$-module. For $x\in \E(M)$ the following
conditions are equivalent:
\begin{enumerate}\setlength\itemsep{0.35em}
\item[\textup{(1)}] $x\in M_w$, that is, there exists $J\in \GV(R)$ such that $Jx\subseteq M$.
\item[\textup{(2)}] For every $\mathfrak m\in w\text{-}\Max(R)$ one has
$x/1\in M_{\mathfrak m}\subseteq \E(M)_{\mathfrak m}$.
\item[\textup{(3)}] There exist $s_1,\dots,s_n\in R$ such that
$J:=(s_1,\dots,s_n)\in \GV(R)$ and $s_i x\in M$ for all $i=1,\dots,n$.
\end{enumerate}
In particular, via the diagonal embedding
\[
\E(M)\hookrightarrow \prod_{\mathfrak m\in w\text{-}\Max(R)}\E(M)_{\mathfrak m},
\]
one has
\[
M_w
=
\E(M)\ \cap\ \prod_{\mathfrak m\in w\text{-}\Max(R)} M_{\mathfrak m}
\subseteq
\prod_{\mathfrak m\in w\text{-}\Max(R)} \E(M)_{\mathfrak m}.
\]
\end{proposition}

\begin{proof}
\textup{(1)$\Rightarrow$(2).}
Assume $x\in M_w$, so there exists $J\in \GV(R)$ such that $Jx\subseteq M$.
Let $\mathfrak m\in w\text{-}\Max(R)$. Since $J\nsubseteq\mathfrak m$,
we may choose $s\in J\setminus\mathfrak m$.
Then $s/1$ is a unit in $R_{\mathfrak m}$, and in $\E(M)_{\mathfrak m}$ we have
\[
(s/1)(x/1)=(sx)/1\in M_{\mathfrak m}.
\]
Because $s/1$ is invertible, it follows that $x/1\in M_{\mathfrak m}$.

\medskip
\noindent
\textup{(2)$\Rightarrow$(3).}
Fix $\mathfrak m\in w\text{-}\Max(R)$. By assumption, $x/1\in M_{\mathfrak m}$,
so there exists $s_{\mathfrak m}\in R\setminus\mathfrak m$ such that
$s_{\mathfrak m}x\in M$.
For each $\mathfrak m$, consider the basic open subset
\[
D(s_{\mathfrak m})\cap w\text{-}\Max(R)
=
\{\mathfrak n\in w\text{-}\Max(R)\mid s_{\mathfrak m}\notin\mathfrak n\}.
\]
These open sets cover $w\text{-}\Max(R)$.
Since the $w$-operation is of finite type, the space $w\text{-}\Max(R)$
is quasi-compact with respect to this basis.
Hence there exist $\mathfrak m_1,\dots,\mathfrak m_n\in w\text{-}\Max(R)$ such that
\[
w\text{-}\Max(R)
\subseteq
D(s_{\mathfrak m_1})\cup\cdots\cup D(s_{\mathfrak m_n}).
\]
Set $s_i:=s_{\mathfrak m_i}$ and $J:=(s_1,\dots,s_n)$.
Then $J\nsubseteq\mathfrak m$ for every $\mathfrak m\in w\text{-}\Max(R)$,
so $J\in \GV(R)$.
By construction, $s_i x\in M$ for all $i=1,\dots,n$.

\medskip
\noindent
\textup{(3)$\Rightarrow$(1).}
This follows immediately from the definition of $M_w$, since $Jx\subseteq M$.

\medskip
\noindent
Finally, the set-theoretic identity follows from the equivalence of
\textup{(1)} and \textup{(2)}, together with the diagonal embedding of $\E(M)$
into the product of its localizations.
\end{proof}

\begin{proposition}\label{prop:envelope-basic}
Let $M$ be a $\GV$-torsion-free $R$-module, and let $M_w$ denote its $w$-envelope
inside the injective hull $\E(M)$.
\begin{enumerate}\setlength\itemsep{0.35em}
\item[\textup{(1)}] \textup{(\cite[Proposition 6.2.4(2)]{WK24})}
$M\subseteq M_w\subseteq \E(M)$ and $M_w/M$ is $\GV$-torsion.

\item[\textup{(2)}]  \textup{(\cite[Theorem 6.2.2(4)]{WK24})}
\emph{Idempotence:} $(M_w)_w=M_w$.

\item[\textup{(3)}] \textup{(\cite[Theorem 6.2.2(4)]{WK24})}
\emph{Minimality:} if $M\subseteq N\subseteq \E(M)$ and $N/M$ is $\GV$-torsion, then
$M_w\subseteq N$.
Equivalently, $M_w$ is the smallest submodule of $\E(M)$ containing $M$
whose quotient modulo $M$ is $\GV$-torsion.

\item[\textup{(4)}]
\emph{Functoriality:} any $R$-linear map $f:M\to N$ between $\GV$-torsion-free modules
extends uniquely to an $R$-linear map
$f_w:M_w\to N_w$ satisfying $f_w|_M=f$.
\end{enumerate}
\end{proposition}

\begin{proof}
\textup{(1)}
By definition, $M\subseteq M_w\subseteq \E(M)$.
Let $x\in M_w$. Then there exists $J\in \GV(R)$ such that $Jx\subseteq M$.
Hence $J(x+M)=0$ in $M_w/M$, so $M_w/M$ is $\GV$-torsion.

\medskip
\noindent
\textup{(2)}
Let $x\in (M_w)_w$. By Proposition~\ref{prop:envelope-equiv}\textup{(2)},
\[
x/1\in (M_w)_{\mathfrak m}
\quad\text{for all }\mathfrak m\in w\text{-}\Max(R).
\]
By \textup{(1)}, the quotient $M_w/M$ is $\GV$-torsion. Hence
\[
(M_w)_{\mathfrak m}=M_{\mathfrak m}
\quad\text{for all }\mathfrak m\in w\text{-}\Max(R).
\]
Therefore $x/1\in M_{\mathfrak m}$ for every $\mathfrak m$.
Applying Proposition~\ref{prop:envelope-equiv}\textup{(2)} (with $M$ in place of $M_w$),
we obtain $x\in M_w$. Thus $(M_w)_w\subseteq M_w$, while the reverse inclusion
is immediate. Hence $(M_w)_w=M_w$.

\medskip
\noindent
\textup{(3)}
Let $M\subseteq N\subseteq \E(M)$ and assume that $N/M$ is $\GV$-torsion.
Since localization kills $\GV$-torsion, we have
\[
N_{\mathfrak m}=M_{\mathfrak m}
\quad\text{for all }\mathfrak m\in w\text{-}\Max(R).
\]
Let $x\in M_w$. By Proposition~\ref{prop:envelope-equiv}\textup{(2)},
\[
x/1\in M_{\mathfrak m}
\quad\text{for all }\mathfrak m\in w\text{-}\Max(R).
\]
Hence $x/1\in N_{\mathfrak m}$ for all such $\mathfrak m$.
Applying Proposition~\ref{prop:envelope-equiv}\textup{(2)} again,
now with $N$ in place of $M$, yields $x\in N_w$.
But $N/M$ is $\GV$-torsion, so $N$ is already $w$-closed inside $\E(M)$,
that is, $N_w=N$.
Therefore $x\in N$, and hence $M_w\subseteq N$.

\medskip
\noindent
\textup{(4)}
Let $f:M\to N$ be an $R$-linear map between $\GV$-torsion-free modules.
Let $f^{E}:\E(M)\to \E(N)$ be an extension to injective hulls.
If $x\in M_w$, choose $J\in \GV(R)$ such that $Jx\subseteq M$.
Then
\[
Jf^{E}(x)=f(Jx)\subseteq f(M)\subseteq N,
\]
so $f^{E}(x)\in N_w$.
Thus $f^{E}$ restricts to a homomorphism
\[
f_w:M_w\to N_w
\]
with $f_w|_M=f$.

For uniqueness, suppose $g:M_w\to N_w$ is another extension with $g|_M=f$.
Then $(g-f_w)|_M=0$.
Since $M$ is an essential submodule of $\E(M)$ and
$M\subseteq M_w\subseteq \E(M)$, it follows that $M$ is essential in $M_w$.
Hence $g-f_w=0$, and therefore $g=f_w$.
\end{proof}

\begin{remark}\label{rem:reconcile}
For a general $R$-module $X$, the categorical reflector associated to the
$\GV$-torsion theory is
\[
F(X)=X/\tor_{\GV}(X),
\]
which is $\GV$-torsion-free and functorial.
Thus $F(-)$ removes $\GV$-torsion by passing to a quotient and is characterized
by its universal property with respect to the full subcategory of
$\GV$-torsion-free modules.

If $X$ is already $\GV$-torsion-free, the $w$-envelope
\[
X_w\subseteq \E(X)
\]
is typically \emph{larger} than $X$; indeed, by
Proposition~\ref{prop:envelope-basic}\textup{(1)}, the quotient
$X_w/X$ is $\GV$-torsion.
In this sense, the constructions
\[
F(-)
\qquad\text{and}\qquad
(-)_w
\]
are complementary: the reflector $F(-)$ \emph{removes} $\GV$-torsion,
whereas the $w$-envelope $(-)_w$ \emph{adds} the minimal enlargement
inside the injective hull whose quotient modulo the original module
is $\GV$-torsion.

Moreover, Proposition~\ref{prop:envelope-basic} shows that on the class
of $\GV$-torsion-free modules the assignment $M\mapsto M_w$ is a closure
operator: it is extensive ($M\subseteq M_w$), idempotent
($ (M_w)_w=M_w$), functorial, and minimal with respect to having
$\GV$-torsion quotient.

In particular, for every $\GV$-torsion-free module $M$,
\[
F(M)=M,
\qquad
M\subseteq M_w\subseteq \E(M),
\qquad
M_w/M\ \text{is $\GV$-torsion}.
\]
\end{remark}

\begin{lemma}
\label{lem:kernel=GV}
Let $R$ be a commutative ring, and let $w$ denote the $w$-operation on $R$.
For any $R$-module $M$, consider the diagonal homomorphism
\[
\delta_M:\ M\longrightarrow 
\prod_{\mathfrak m\in w\text{-}\Max(R)} M_{\mathfrak m},
\qquad
x\longmapsto ((x/1)_{\mathfrak m}).
\]
Then
\[
\ker(\delta_M)=\tor_{\GV}(M).
\]
\end{lemma}

\begin{proof}
Apply \cite[Chapter~IX, Lemma~6.3]{FS01} with 
$\mathcal{P}=w\text{-}\Max(R)$.
\end{proof}

\begin{corollary}
\label{cor:w-closure=quot}
For every $R$-module $M$, the diagonal homomorphism
\[
\delta_M:\ M\longrightarrow 
\prod_{\mathfrak m\in w\text{-}\Max(R)} M_{\mathfrak m}
\]
factors canonically as
\[
M\ \twoheadrightarrow\ M/\tor_{\GV}(M)
\ \xrightarrow{\ \overline{\delta}_M\ }\ 
\prod_{\mathfrak m\in w\text{-}\Max(R)} M_{\mathfrak m},
\]
where $\overline{\delta}_M$ is injective. In particular, there is a natural
isomorphism
\[
\Im(\delta_M)\ \cong\ M/\tor_{\GV}(M).
\]
\end{corollary}

\begin{proof}
By Lemma~\ref{lem:kernel=GV}, we have
$\ker(\delta_M)=\tor_{\GV}(M)$.
Hence $\delta_M$ induces an injective homomorphism
\[
\overline{\delta}_M:\ 
M/\tor_{\GV}(M)\hookrightarrow
\prod_{\mathfrak m\in w\text{-}\Max(R)} M_{\mathfrak m}.
\]
By construction, the image of $\overline{\delta}_M$ coincides with
$\Im(\delta_M)$, yielding the claimed identification.
\end{proof}

\section{The category $(R,w)\text{-}\mathrm{mod}$ and $w$-faithfully exact functors}

Let $R$ be a commutative ring and let $w$ denote the $w$-operation on $R$. 
Set
\[
\mathcal T(R)
:=\bigl\{\,M\in\Mod(R)\mid 
M_{\mathfrak m}=0
\text{ for all }\mathfrak m\in w\text{-}\Max(R)\,\bigr\},
\]
the class of $\GV$-torsion $R$-modules, and
\[
\mathcal F(R)
:=\{\,M\in\Mod(R)\mid \tor_{\GV}(M)=0\,\},
\]
the class of $\GV$-torsion-free $R$-modules.
Then $\tau_w=(\mathcal T(R),\mathcal F(R))$ is a hereditary torsion theory of
\emph{finite type} on $\Mod(R)$ (see \cite{S75}).

\begin{definition}
\label{def:Rw-mod}
We define $(R,w)\text{-}\mathrm{mod}$ to be the full subcategory of 
$\Mod(R)$ consisting of all $w$-modules.
\end{definition}

We record the following basic properties of $(R,w)\text{-}\mathrm{mod}$:

\begin{enumerate}
\item $(R,w)\text{-}\mathrm{mod}$ is reflective; that is, the inclusion
\[
i:(R,w)\text{-}\mathrm{mod}\hookrightarrow \Mod(R)
\]
admits a left adjoint
\[
L:\Mod(R)\longrightarrow (R,w)\text{-}\mathrm{mod}.
\]

\item $(R,w)\text{-}\mathrm{mod}$ is a Grothendieck category.

\item The reflector $L$ is exact and annihilates $\mathcal T(R)$.

\item The functor $L$ is universal among exact functors on $\Mod(R)$
that annihilate $\mathcal T(R)$. More precisely, given a diagram
\[
\xymatrix{
\Mod(R)\ar[rr]^G\ar[rd]_L && \mathcal A\\
& (R,w)\text{-}\mathrm{mod}\ar@{-->}[ru]_F &
}
\]
where $G$ is exact and annihilates $\mathcal T(R)$,
there exists a unique (up to natural isomorphism) exact functor
$F:(R,w)\text{-}\mathrm{mod}\to\mathcal A$
such that $F \circ L$ is naturally isomorphic to $G$.
In fact, one may take $F=G|_{(R,w)\text{-}\mathrm{mod}}$.
\end{enumerate}

\begin{remark}
The exact structure on $(R,w)\text{-}\mathrm{mod}$ is inherited from
$\Mod(R)$ via the reflector. A sequence
\[
0\longrightarrow A'\longrightarrow A\longrightarrow A''\longrightarrow 0
\]
of $w$-modules is exact in $(R,w)\text{-}\mathrm{mod}$
if and only if it becomes exact after applying the quotient functor $T$,
equivalently, after localizing at every $\mathfrak m\in w\text{-}\Max(R)$.
\end{remark}

We briefly recall the construction of the Serre quotient
(see, for example, \cite[p.~42]{S68}).
The reflector $L:\Mod(R)\to (R,w)\text{-}\mathrm{mod}$
induces an equivalence between $(R,w)\text{-}\mathrm{mod}$
and the Serre quotient category
\[
\Mod(R)/\mathcal T(R).
\]
We denote this quotient by
\[
\mathsf Q_w(R):=\Mod(R)/\mathcal T(R),
\]
and call it the \emph{$w$-quotient category} of $\Mod(R)$.
Its objects coincide with those of $\Mod(R)$, and
\[
\Hom_{\mathsf Q_w(R)}(M,N)
=
\Hom_{(R,w)\text{-}\mathrm{mod}}(L(M),L(N))
=
\Hom_R(L(M),L(N)).
\]

We have a commutative diagram
\[
\xymatrix{
\Mod(R)\ar[rr]^L\ar[rd]_T && (R,w)\text{-}\mathrm{mod}\\
& \mathsf Q_w(R)\ar[ru]_S &
}
\]
where $S$ is an equivalence, $T$ is the canonical quotient functor,
and $S \circ T=L$.
Explicitly, $T$ acts as the identity on objects, and on morphisms
\[
T:\Hom_R(M,N)\longrightarrow
\Hom_{\mathsf Q_w(R)}(M,N)
=
\Hom_R(L(M),L(N)).
\]

Moreover:

\begin{enumerate}
\item $T$ is exact and annihilates $\mathcal T(R)$.
\item $\mathsf Q_w(R)$ is an abelian category.
\item The functor $T:\Mod(R)\to\mathsf Q_w(R)$
is universal among exact functors on $\Mod(R)$
that annihilate $\mathcal T(R)$.
That is, if $G:\Mod(R)\to\mathcal A$ is exact and
annihilates $\mathcal T(R)$, then there exists a unique exact functor
\[
\overline G:\mathsf Q_w(R)\to\mathcal A
\]
such that $\overline G \circ T=G$.
\end{enumerate}

For $M,N\in\Mod(R)$, define
\[
\hom_w(M,N)
:=
\Hom_{\mathsf Q_w(R)}\!\big(T(M),T(N)\big).
\]

\begin{remark}[On $T(M)$ versus $L(M)$]\label{rem:TM-vs-LM}
For every $M\in\Mod(R)$, there is a canonical isomorphism
\[
T(c_M):T(M)\xrightarrow{\ \sim\ }T\!\big(L(M)\big),
\]
induced by the canonical morphism
\[
c_M:\ 
M\twoheadrightarrow M/\tor_{\GV}(M)
\hookrightarrow (M/\tor_{\GV}(M))_w.
\]
Indeed, $\ker(c_M)=\tor_{\GV}(M)\in\mathcal T(R)$ and
$\coker(c_M)\in\mathcal T(R)$, so $T(c_M)$ is an isomorphism
in $\mathsf Q_w(R)$.
\end{remark}

Recall from \cite{I64} that a covariant or contravariant additive functor $G$
between abelian categories is called \emph{faithfully exact}
(resp.\ \emph{exact}) if, for every composable pair
$A\xrightarrow{f}B\xrightarrow{g}C$, the sequence
\[
G(A)\longrightarrow G(B)\longrightarrow G(C)
\]
is exact if and only if (resp.\ if)
\[
A\longrightarrow B\longrightarrow C
\]
is exact.

In an abelian category, an additive functor is faithfully exact if and only if
it is exact and faithful (i.e.\ injective on $\Hom$-sets, or equivalently,
it detects nonzero objects).

\begin{definition}
\label{def:w-faithful-functor}
Keep the notation as above. Let $\mathcal A$ be an abelian category.
A covariant or contravariant additive functor
\[
G:\Mod(R)\longrightarrow\mathcal A
\]
is called \emph{$w$-faithfully exact} if the following conditions hold:
\begin{enumerate}\setlength\itemsep{0.25em}
\item[\textup{(i)}] $G$ is exact and annihilates $\mathcal T(R)$;
\item[\textup{(ii)}] the induced functor
\[
\overline G:\mathsf Q_w(R)\longrightarrow\mathcal A
\]
is faithfully exact.
\end{enumerate}
\end{definition}

\begin{remark}
\label{rem:faithfully-exact-vs-w}
There is no conceptual difference between the two notions:
a $w$-faithfully exact functor is precisely an Ishikawa faithfully exact
functor formulated inside the $w$-quotient category $\mathsf Q_w(R)$.
\end{remark}

\begin{proposition}
\label{prop:w-faithful-equivalences}
Let $G:\Mod(R)\to\mathcal A$ be an exact functor that annihilates
$\mathcal T(R)$. The following conditions are equivalent:
\begin{enumerate}\setlength\itemsep{0.35em}
\item[\textup{(1)}]
$G$ is $w$-faithfully exact.

\item[\textup{(2)}]
The restriction
\[
G|_{(R,w)\text{-}\mathrm{mod}}:
(R,w)\text{-}\mathrm{mod}\longrightarrow \mathcal A
\]
is faithfully exact.

\item[\textup{(3)}]
The composite functor
\[
G|_{(R,w)\text{-}\mathrm{mod}}\circ S:
\mathsf Q_w(R)\longrightarrow \mathcal A,
\]
where $S:\mathsf Q_w(R)\to (R,w)\text{-}\mathrm{mod}$
is the equivalence, is faithfully exact.
\end{enumerate}
\end{proposition}

\begin{proof}
(1)$\Leftrightarrow$(3).
Since faithful exactness is preserved under natural isomorphisms of functors, it suffices to show that the induced functor
\[
\overline G:\mathsf Q_w(R)\to\mathcal A
\]
is naturally isomorphic to the composite functor
\[
G|_{(R,w)\text{-}\mathrm{mod}}\circ S:\mathsf Q_w(R)\longrightarrow \mathcal A.
\]

Let $f\in\Hom_{\mathsf Q_w(R)}(M,N)=\Hom_R(L(M),L(N))$. 
Then in $\Mod(R)$ we have
\[
\begin{array}{l}
	\xymatrix{
		M\ar[d]_{c_M} & N\ar[d]^{c_N} \\
		L(M)\ar[r]_{f} & L(N)
	}
\end{array}
\eqno{\textcircled{1}}
\]

Applying $T$ to \textcircled{1} yields
\[
\begin{array}{l}
	\xymatrix{
		T(M)\ar[r]^{f}\ar[d]_{T(c_M)}^{\cong} & T(N)\ar[d]^{T(c_N)}_{\cong} \\
		T \circ L(M)\ar[r]_{T(f)} & T \circ L(N)
	}
\end{array}
\eqno{\textcircled{2}}
\]
in $\mathsf Q_w(R)$, where $T(c_M)$ and $T(c_N)$ are isomorphisms.

Applying $\overline G$ and $G|_{(R,w)\text{-}\mathrm{mod}}\circ S$ to \textcircled{2} gives

\[
\begin{array}{l}
	\xymatrix@C=2cm{
		\overline{G} \circ T(M)\ar[r]^{\overline{G}(f)}\ar[d]_{\overline{G} \circ T(c_M)}^{\cong} 
		& \overline{G} \circ T(N)\ar[d]^{\overline{G} \circ T(c_N)}_{\cong} \\
		\overline{G} \circ T \circ L(M)\ar[r]_{\overline{G} \circ T(f)} & \overline{G} \circ T \circ L(N)
	}
\end{array}
\eqno{\textcircled{3}}
\]

and

\[
\begin{array}{l}
	\xymatrix@C=3cm{
		G|_{(R,w)\text{-}\mathrm{mod}} \circ S \circ T(M)
		\ar[r]^{G|_{(R,w)\text{-}\mathrm{mod}} \circ S(f)}
		\ar[d]_{G|_{(R,w)\text{-}\mathrm{mod}} \circ S \circ T(c_M)}^{\cong} 
		& G|_{(R,w)\text{-}\mathrm{mod}} \circ S \circ T(N)
		\ar[d]^{G|_{(R,w)\text{-}\mathrm{mod}} \circ S \circ T(c_N)}_{\cong} \\
		G|_{(R,w)\text{-}\mathrm{mod}} \circ S \circ T \circ L(M)
		\ar[r]_{G|_{(R,w)\text{-}\mathrm{mod}} \circ S \circ T(f)} 
		& G|_{(R,w)\text{-}\mathrm{mod}} \circ S \circ T \circ L(N)
	}
\end{array}
\eqno{\textcircled{4}}
\]
in $\mathcal A$.

Note that $G$ is naturally isomorphic to $G|_{(R,w)\text{-}\mathrm{mod}}\circ L$, 
that $G=\overline G\circ T$, and that $L=S\circ T$. 
Hence $\overline G\circ T$ is naturally isomorphic to 
$G|_{(R,w)\text{-}\mathrm{mod}}\circ S\circ T$. 
Therefore we obtain

\[
\begin{array}{l}
	\xymatrix@C=3cm{
		\overline{G} \circ T \circ L(M)\ar[r]^{\overline{G} \circ T(f)}\ar[d]^{\cong} 
		& \overline{G} \circ T \circ L(N)\ar[d]_{\cong} \\
		G|_{(R,w)\text{-}\mathrm{mod}} \circ S \circ T \circ L(M)
		\ar[r]_{G|_{(R,w)\text{-}\mathrm{mod}} \circ S \circ T(f)} 
		& G|_{(R,w)\text{-}\mathrm{mod}} \circ S \circ T \circ L(N)
	}
\end{array}
\eqno{\textcircled{5}}
\]

Combining \textcircled{3}, \textcircled{4}, and \textcircled{5}, we obtain

\[
\xymatrix@C=3cm{
	\overline{G} \circ T(M)\ar[r]^{\overline{G}(f)}\ar[d]^{\cong} 
	& \overline{G} \circ T(N)\ar[d]_{\cong} \\
	G|_{(R,w)\text{-}\mathrm{mod}} \circ S \circ T(M)
	\ar[r]_{G|_{(R,w)\text{-}\mathrm{mod}} \circ S(f)} 
	& G|_{(R,w)\text{-}\mathrm{mod}} \circ S \circ T(N)
}
\]

Since $T(M)=M$ and $T(N)=N$, this becomes

\[
\xymatrix@C=3cm{
	\overline{G}(M)\ar[r]^{\overline{G}(f)}\ar[d]^{\cong} 
	& \overline{G}(N)\ar[d]_{\cong} \\
	G|_{(R,w)\text{-}\mathrm{mod}} \circ S(M)
	\ar[r]_{G|_{(R,w)\text{-}\mathrm{mod}} \circ S(f)} 
	& G|_{(R,w)\text{-}\mathrm{mod}} \circ S(N)
}
\]

Thus $\overline G$ is naturally isomorphic to 
$G|_{(R,w)\text{-}\mathrm{mod}}\circ S$, as required.

\smallskip
\noindent
(2)$\Leftrightarrow$(3).
This is immediate since $S$ is an equivalence of categories.
\end{proof}

\noindent\textbf{Working criterion (abelian setting).}
For additive functors between abelian categories,
\[
\text{faithfully exact}
\quad\Longleftrightarrow\quad
\text{exact and faithful}.
\]
Hence, in the $w$-context, a functor $G$ is $w$-faithfully exact if and only if
its restriction to $(R,w)\text{-}\mathrm{mod}$ is exact and detects nonzero
objects and nonzero morphisms there.

\begin{theorem}$($$w$-version of Ishikawa's \cite[Theorem~1.1]{I64}$)$
\label{thm:w-Ishi-11}
Let $\mathcal A$ be an abelian category and let
$G:\Mod(R)\to\mathcal A$ be an exact covariant functor
that annihilates $\mathcal T(R)$.
Then the following conditions are equivalent:
\begin{enumerate}\setlength\itemsep{0.35em}
\item[\textup{(1)}] $G$ is $w$-faithfully exact.
\item[\textup{(2)}] $G(A)\neq 0$ for every nonzero $w$-module $A$.
\item[\textup{(3)}] $G(\varphi)\neq 0$ for every nonzero morphism
$\varphi$ in $(R,w)\text{-}\mathrm{mod}$.
\item[\textup{(4)}] $G(B)\neq 0$ for every nonzero $\GV$-torsion-free
$R$-module $B$.
\item[\textup{(5)}] $G(R/I)\neq 0$ for every proper $w$-ideal $I$ of $R$.
\item[\textup{(6)}] $G(R/\mathfrak m)\neq 0$ for every
$\mathfrak m\in w\text{-}\Max(R)$.
\end{enumerate}
For a contravariant exact functor, the same equivalences hold with the
obvious reversal of arrow directions.
\end{theorem}

\begin{proof}
(1)$\Leftrightarrow$(2)$\Leftrightarrow$(3).
By Proposition~\ref{prop:w-faithful-equivalences},
$G$ is $w$-faithfully exact if and only if its restriction
\[
G|_{(R,w)\text{-}\mathrm{mod}}:
(R,w)\text{-}\mathrm{mod}\longrightarrow \mathcal A
\]
is faithfully exact.
The equivalence of \textup{(1)}--\textup{(3)} therefore follows from
\cite[Sublemma~1.15]{S68}.

\medskip
\noindent
(2)$\Rightarrow$(4).
Assume \textup{(2)} and let $B$ be a nonzero $\GV$-torsion-free
$R$-module.
Then the sequence
\[
0\longrightarrow B\longrightarrow B_w
\longrightarrow B_w/B\longrightarrow 0
\]
is exact, with $B_w\neq 0$ and $B_w/B\in\mathcal T(R)$.
Since $G$ is exact and annihilates $\mathcal T(R)$, we obtain an exact
sequence
\[
0\longrightarrow G(B)\longrightarrow G(B_w)\longrightarrow 0
\]
in $\mathcal A$.
Thus $G(B)\cong G(B_w)$.
By \textup{(2)}, $G(B_w)\neq 0$, hence $G(B)\neq 0$.

\medskip
\noindent
(4)$\Rightarrow$(5).
If $I$ is a proper $w$-ideal, then $R/I$ is $\GV$-torsion-free and
nonzero. Hence $G(R/I)\neq 0$ by \textup{(4)}.

\medskip
\noindent
(5)$\Rightarrow$(6).
Every $\mathfrak m\in w\text{-}\Max(R)$ is a proper $w$-ideal.
Thus \textup{(6)} follows immediately from \textup{(5)}.

\medskip
\noindent
(6)$\Rightarrow$(2).
Assume \textup{(6)} and let $A$ be a $w$-module with $G(A)=0$.
We show that $A=0$.

For any $a\in A$, consider the cyclic submodule $Ra$.
Since $G$ is exact, $G(Ra)=0$.
Observe that $Ra\cong R/\ann_R(a)$.
Because $A$ is a $w$-module, $\ann_R(a)$ is a $w$-ideal.

If $\ann_R(a)\neq R$, then there exists
$\mathfrak m\in w\text{-}\Max(R)$ with
$\ann_R(a)\subseteq \mathfrak m$.
Hence we have a surjection
\[
R/\ann_R(a)\twoheadrightarrow R/\mathfrak m.
\]
Applying $G$ yields an exact sequence
\[
0=G(R/\ann_R(a))
\longrightarrow G(R/\mathfrak m)
\longrightarrow 0,
\]
so $G(R/\mathfrak m)=0$, contradicting \textup{(6)}.
Therefore $\ann_R(a)=R$, so $a=0$.
Thus $A=0$.
\end{proof}

\section{$w$-faithfully projective modules}

Fix the $w$-operation on a commutative ring $R$ and work in the Grothendieck category
$(R,w)\text{-}\mathrm{mod}$
of $w$-modules. Let
\[
L(-):=((-)/\tor_{\GV}(-))_w
\]
denote the reflector onto $(R,w)\text{-}\mathrm{mod}$.

For a fixed $R$-module $P$, define a covariant functor on
$\Mod(R)$ by
\[
H_w(X)
:=\hom_w(P,X)
=\Hom_R\!\big(L(P),\,L(X)\big),
\qquad X\in\Mod(R).
\]
Thus we obtain a left exact functor
\[
H_w(-):
\Mod(R)\xrightarrow{L(-)}
(R,w)\text{-}\mathrm{mod}
\xrightarrow{\Hom_R(L(P),-)}
(R,w)\text{-}\mathrm{mod},
\]
which annihilates $\mathcal T(R)$.

\medskip

Let
\[
\xi:\ 0 \longrightarrow A
\xrightarrow{f}
B
\xrightarrow{g}
C
\longrightarrow 0
\]
be a short exact sequence of $R$-modules.
Recall from \cite{WQ20} that $\xi$ is said to be \emph{$w$-split}
if there exist $J=(d_1,\dots,d_n)\in\GV(R)$ and
homomorphisms $h_1,\dots,h_n\in\Hom_R(C,B)$ such that
\[
d_k\cdot 1_C = g h_k,
\qquad k=1,\dots,n.
\]
In this case, $J$ is called a $\GV$-ideal associated with the
$w$-split exact sequence $\xi$.

An $R$-module $M$ is called \emph{$w$-split}
if there exist a projective module $F$ and an epimorphism
$g:F\to M$ such that the exact sequence
\[
0\longrightarrow \ker(g)
\longrightarrow F
\longrightarrow M
\longrightarrow 0
\]
is $w$-split.

Let $M$ and $N$ be $R$-modules and let $f:M\to N$ be an $R$-homomorphism. 
We say that $f$ is a \emph{$w$-isomorphism} if both $\ker(f)$ and $\coker(f)$ are $\GV$-torsion modules. 

Equivalently, $f$ is a $w$-isomorphism if the induced map
\[
L(f):L(M)\longrightarrow L(N)
\]
is an isomorphism in $(R,w)\text{-}\mathrm{mod}$ (equivalently, in $\Mod(R)$).

We now present necessary and sufficient conditions for the exactness
of the functor $H_w(-)$.

\begin{lemma}\label{lem:w-pro-w-split}
Let $M$ be an $R$-module. Then $M$ is $w$-projective
if and only if $L(M)$ is $w$-split.
\end{lemma}

\begin{proof}
Since the canonical map $c_M:M\to L(M)$ is a $w$-isomorphism,
$M$ is $w$-projective if and only if $L(M)$ is $w$-projective
by \cite[Proposition~2.3(1)]{WK15}.
Moreover, $L(M)$ is a $w$-module, and
\cite[Theorem~6.7.11]{WK24} shows that a $w$-module is $w$-projective
if and only if it is $w$-split.
Hence $M$ is $w$-projective if and only if $L(M)$ is $w$-split.
\end{proof}

\begin{proposition}\label{prop:exactness of H_w}
If the functor
\[
H_w(-):\Mod(R)\longrightarrow (R,w)\text{-}\mathrm{mod}
\]
is exact, then $L(P)$ is $w$-split.
Equivalently, $P$ is $w$-projective.
\end{proposition}

\begin{proof}
By Lemma~\ref{lem:w-pro-w-split}, it suffices to show that
$L(P)$ is $w$-split.
Let $g:F\to L(P)$ be an epimorphism with $F$ projective.
The exactness of $H_w(-)$ applied to
\[
0\longrightarrow \ker(g)
\longrightarrow F
\longrightarrow L(P)
\longrightarrow 0
\]
yields an exact sequence
\[
\Hom_R\big(L(P),F\big)
\xrightarrow{g_*}
\Hom_R\big(L(P),L(P)\big)
\longrightarrow 0
\]
in $(R,w)\text{-}\mathrm{mod}$.
Thus $\coker(g_*)$ is $\GV$-torsion.
By \cite[Proposition~2.4]{WQ20},
this implies that $L(P)$ is $w$-split.
\end{proof}

\begin{remark}\label{rem:standing}
Let $R$ be a commutative ring and let $P$ be a $w$-projective
$R$-module of \emph{finite type}, that is, there exists a $w$-exact
sequence
\[
F\longrightarrow P\longrightarrow 0,
\]
where $F$ is finitely generated free.
Then \cite[Theorem~2.19]{WK15} shows that $P$ is of
\emph{finitely presented type}, meaning that there exists a
$w$-exact sequence
\[
F_1\longrightarrow F_0\longrightarrow P\longrightarrow 0,
\]
where $F_0$ and $F_1$ are finitely generated free modules.
Consequently, by \cite[Theorem~4.6]{QW15}, under these hypotheses the functor
\[
H_w(-)=\hom_w(P,-):
\Mod(R)\longrightarrow (R,w)\text{-}\mathrm{mod}
\]
is exact and annihilates $\mathcal T(R)$.
\end{remark}

\begin{proposition}
\label{prop:F(P)-proj-classical-1}
Let $R$ be a commutative ring and let
$P\in (R,w)\text{-}\mathrm{mod}$.
Then the following conditions are equivalent:
\begin{enumerate}
\item[\textup{(1)}]
$P$ is a projective object of $(R,w)\text{-}\mathrm{mod}$.

\item[\textup{(2)}]
For every morphism $f:B\to C$ in $\Mod(R)$
with $\coker(f)\in\mathcal T(R)$,
the induced map
\[
\Hom_R(P,L(B))
\longrightarrow
\Hom_R(P,L(C))
\]
is surjective in $(R,w)\text{-}\mathrm{mod}$.
Equivalently, every morphism $P\to L(C)$ lifts along $L(f)$.

\item[\textup{(3)}]
The functor
\[
H_w(-)=\Hom_R(P,L(-))
\]
is exact.
\end{enumerate}
\end{proposition}

\begin{proof}
\noindent
(1)$\Rightarrow$(2).
Let $f:B\to C$ be a morphism in $\Mod(R)$
with $\coker(f)\in\mathcal T(R)$.
Applying $L$ yields an exact sequence
\[
L(B)\longrightarrow L(C)\longrightarrow L(\coker(f))=0
\]
in $(R,w)\text{-}\mathrm{mod}$.
Since $P$ is projective in $(R,w)\text{-}\mathrm{mod}$,
the induced map
\[
\Hom_R(P,L(B))
\longrightarrow
\Hom_R(P,L(C))
\]
is surjective.

\medskip
\noindent
(2)$\Rightarrow$(3).
Let
\[
0\longrightarrow A
\longrightarrow B
\longrightarrow C
\longrightarrow 0
\]
be an exact sequence in $\Mod(R)$.
Applying $L$ yields an exact sequence in $(R,w)\text{-}\mathrm{mod}$.
By the lifting property, the induced sequence
\[
0\longrightarrow
\Hom_R(P,L(A))
\longrightarrow
\Hom_R(P,L(B))
\longrightarrow
\Hom_R(P,L(C))
\longrightarrow 0
\]
is exact.
Hence $H_w(-)$ is exact.

\medskip
\noindent
(3)$\Rightarrow$(1).
Let
\[
B\longrightarrow C\longrightarrow 0
\]
be an exact sequence in $(R,w)\text{-}\mathrm{mod}$.
Since $H_w(-)$ is exact on $\Mod(R)$,
its restriction to $(R,w)\text{-}\mathrm{mod}$
is also exact.
Thus
\[
\Hom_R(P,B)
\longrightarrow
\Hom_R(P,C)
\longrightarrow 0
\]
is exact in $(R,w)\text{-}\mathrm{mod}$,
so $P$ is projective in $(R,w)\text{-}\mathrm{mod}$.
\end{proof}

\begin{corollary}\label{cor:projective-obj}
Every $w$-projective $w$-module of finite type
(for example, a finitely generated projective $R$-module)
is a projective object in $(R,w)\text{-}\mathrm{mod}$.
\end{corollary}

\begin{definition}
\label{def:w-faith-proj}
Let $R$ be a commutative ring.
An $R$-module $P$ is called \emph{$w$-faithfully projective}
if the functor
\[
H_w(-)=\hom_w(P,-):
\Mod(R)\longrightarrow (R,w)\text{-}\mathrm{mod}
\]
is $w$-faithfully exact.
Equivalently, the object $L(P)$ is faithfully projective
in the Grothendieck category $(R,w)\text{-}\mathrm{mod}$.
\end{definition}

\begin{remark}\label{rem:w-faith-pro-closed-w-iso}
Since
\[
\hom_w(P,-)=\Hom_R(L(P),L(-)),
\]
the notion of $w$-faithfully projective modules depends only on the
$w$-envelope $L(P)$.
In particular, replacing $P$ by any module having the same image under
$L(-)$ does not affect this property.
For example, if $f:P\to P'$ is a $w$-isomorphism of $R$-modules,
then $L(P)\cong L(P')$,
and hence $P$ is $w$-faithfully projective if and only if
$P'$ is $w$-faithfully projective.
\end{remark}

\begin{proposition}
\label{prop:w-faith-projective-local}
Let $\mathfrak m$ range over $w\text{-}\Max(R)$ and let
$P$ be an $R$-module of finitely presented type.
Then $P$ is $w$-faithfully projective if and only if
\begin{enumerate}
\item[\textup{(i)}] $L(P)$ is projective in $(R,w)\text{-}\mathrm{mod}$, and
\item[\textup{(ii)}] $P_{\mathfrak m}\neq 0$ for every $\mathfrak m\in w\text{-}\Max(R)$.
\end{enumerate}
The latter condition is equivalent to
\[
\Hom_{R_{\mathfrak m}}\!\big(P_{\mathfrak m},
R_{\mathfrak m}/\mathfrak mR_{\mathfrak m}\big)\neq 0
\quad\text{for all }\mathfrak m\in w\text{-}\Max(R).
\]
In particular, it suffices that $P$ is $w$-projective as an $R$-module and
\[
\Supp_R(P)\supseteq w\text{-}\Max(R).
\]
\end{proposition}

\begin{proof}
Let
\[
H_w(-)=\hom_w(P,-):
\Mod(R)\longrightarrow (R,w)\text{-}\mathrm{mod}.
\]

\medskip
\noindent
\emph{Step 1: Exactness of $H_w$ and projectivity of $L(P)$.}

For $X\in (R,w)\text{-}\mathrm{mod}$ we have $L(X)=X$, hence
\[
H_w(X)=\hom_w(P,X)=\Hom_R\!\big(L(P),X\big).
\]
Thus $H_w$ is exact on $(R,w)\text{-}\mathrm{mod}$
if and only if the functor
$\Hom_R(L(P),-)$
is exact on $(R,w)\text{-}\mathrm{mod}$,
which holds if and only if $L(P)$ is projective in
$(R,w)\text{-}\mathrm{mod}$.

\medskip
\noindent
\emph{Step 2: Faithfulness of $H_w$ and testing on residue objects.}

Assume that $H_w$ is exact.
By Theorem~\ref{thm:w-Ishi-11},
$H_w$ is $w$-faithfully exact if and only if
\[
H_w(R/\mathfrak m)\neq 0
\quad\text{for all }\mathfrak m\in w\text{-}\Max(R).
\]
Using the description of $w$-Hom,
\[
H_w(R/\mathfrak m)
=\hom_w(P,R/\mathfrak m)
=\Hom_R\!\big(L(P),(R/\mathfrak m)_w\big).
\]

Since $(R/\mathfrak m)_w$ is supported only at $\mathfrak m$,
the above group is nonzero if and only if its localization at
$\mathfrak m$ is nonzero:
\[
\Hom_R\!\big(L(P),(R/\mathfrak m)_w\big)_{\mathfrak m}\neq 0.
\]

Because $P$ and $L(P)$ are $w$-isomorphic,
$L(P)$ is also of finitely presented type.
Moreover,
\[
L(P)_{\mathfrak m}\cong P_{\mathfrak m},
\qquad
\big((R/\mathfrak m)_w\big)_{\mathfrak m}
\cong R_{\mathfrak m}/\mathfrak mR_{\mathfrak m}.
\]
Hence, by \cite[Theorem~6.3.28(2)]{WK24},
\[
\Hom_R\!\big(L(P),(R/\mathfrak m)_w\big)_{\mathfrak m}
\cong
\Hom_{R_{\mathfrak m}}\!\big(
P_{\mathfrak m},
R_{\mathfrak m}/\mathfrak mR_{\mathfrak m}
\big).
\]

Therefore,
\[
H_w(R/\mathfrak m)\neq 0
\quad\Longleftrightarrow\quad
\Hom_{R_{\mathfrak m}}\!\big(
P_{\mathfrak m},
R_{\mathfrak m}/\mathfrak mR_{\mathfrak m}
\big)\neq 0.
\]

Since $R_{\mathfrak m}$ is local,
the latter condition holds if and only if
$P_{\mathfrak m}\neq 0$.
This establishes the equivalence.

\medskip
\noindent
\emph{Step 3: The ``in particular'' statement.}

If $P$ is $w$-projective as an $R$-module,
then by Remark~\ref{rem:standing} the functor $H_w(-)$ is exact.
Moreover,
\[
\mathfrak m\in\Supp_R(P)
\quad\Longleftrightarrow\quad
P_{\mathfrak m}\neq 0.
\]
Thus the condition
$\Supp_R(P)\supseteq w\text{-}\Max(R)$
implies that $P$ is $w$-faithfully projective.
\end{proof}

\begin{corollary}\label{cor:examples-of-w-faith-pro}
Every nonzero finitely generated free $R$-module is $w$-faithfully projective.
\end{corollary}

\begin{proof}
Let $P$ be a nonzero finitely generated free $R$-module.
Then $P$ is projective and of finite type, hence $w$-projective.
By Corollary~\ref{cor:projective-obj}, $P$ is a projective object in
$(R,w)\text{-}\mathrm{mod}$.

Moreover, since $P\cong R^n$ for some $n\ge 1$, we have
\[
\Supp_R(P)=\Spec(R).
\]
In particular,
\[
\Supp_R(P)\supseteq w\text{-}\Max(R).
\]
Therefore, by Proposition~\ref{prop:w-faith-projective-local},
$P$ is $w$-faithfully projective.
\end{proof}

\begin{theorem}\textup{($w$-version of Ishikawa's \cite[Theorem~2.2]{I64})}
\label{thm:w-Ishi-22}
Let $R$ be a commutative ring and let $P$ be a $w$-projective
$R$-module of finite type.
Then the following conditions are equivalent:
\begin{enumerate}
\item[\textup{(1)}]
$P$ is \emph{$w$-faithfully projective}.

\item[\textup{(2)}]
$\hom_w(P,A)\neq 0$ for every nonzero $w$-module $A$.

\item[\textup{(3)}]
For every nonzero morphism $\varphi:C\to D$
in $(R,w)\text{-}\mathrm{mod}$,
the induced map
\[
\Hom_R(L(P),C)\xrightarrow{\ \varphi\circ -\ }
\Hom_R(L(P),D)
\]
is nonzero; equivalently,
there exists $\psi\in\Hom_R(L(P),C)$
such that $\varphi\psi\neq 0$.

\item[\textup{(4)}]
$\hom_w(P,B)\neq 0$
for every nonzero $\GV$-torsion-free $R$-module $B$.

\item[\textup{(5)}]
$\hom_w(P,R/I)\neq 0$
for every proper $w$-ideal $I$ of $R$.

\item[\textup{(6)}]
$\hom_w(P,R/\mathfrak m)\neq 0$
for every $\mathfrak m\in w\text{-}\Max(R)$.

\item[\textup{(7)}]
\emph{(Local criterion)}
$P_{\mathfrak m}\neq 0$
for every $\mathfrak m\in w\text{-}\Max(R)$;
equivalently,
\[
w\text{-}\Max(R)\subseteq \Supp_R(P).
\]
\end{enumerate}
\end{theorem}

\begin{proof}
By Remark~\ref{rem:standing},
the contravariant functor
\[
H_w(-)=\hom_w(P,-)
\]
is exact and annihilates $\mathcal T(R)$.
Hence the $w$-version of Ishikawa's Theorem~1.1
(Theorem~\ref{thm:w-Ishi-11})
applies to $H_w(-)$ and yields the equivalences
\[
(1)
\Longleftrightarrow
(2)
\Longleftrightarrow
(3)
\Longleftrightarrow
(4)
\Longleftrightarrow
(5)
\Longleftrightarrow
(6).
\]
The equivalence of
(6)
and
(7)
follows from Proposition~\ref{prop:w-faith-projective-local}.
\end{proof}

\begin{example}[A nonzero finitely presented $w$-projective module that is not $w$-faithfully projective]
\label{exa:non-w-faith-pro}
Let $R$ be a commutative ring that is not a DW-ring.
Then there exists a proper $\GV$-ideal $J\subsetneq R$.
Set $Q:=R/J$.
Then $Q$ is a nonzero finitely presented $\GV$-torsion $R$-module.
In particular, $Q$ is $w$-projective.

However, since $Q$ is $\GV$-torsion, we have
\[
Q_{\mathfrak m}=0
\quad\text{for all }
\mathfrak m\in w\text{-}\Max(R).
\]
Thus
\[
w\text{-}\Max(R)\cap\Supp_R(Q)=\varnothing,
\]
so
\[
w\text{-}\Max(R)\nsubseteq \Supp_R(Q).
\]
By Theorem~\ref{thm:w-Ishi-22},
$Q$ is not $w$-faithfully projective.
\end{example}

\begin{remark}
\label{rem:LR=R}
Since $R$ is a $w$-module over itself,
we have $L(R)=R$.
In particular,
\[
\hom_w(P,R)=\Hom_R\!\big(L(P),R\big),
\]
and we may define the $w$-trace of $P$ by
\[
\tr_w(P)
:=
\sum_{f\in\hom_w(P,R)}
f\big(L(P)\big)
\subseteq R.
\]
\end{remark}

\begin{proposition}\textup{($w$-analogue of Ishikawa's \cite[Proposition~2.4]{I64})}
\label{prop:w-Ishikawa24}
With $R$, $P$, and $H_w(-)$ as in Remark~\ref{rem:standing},
the following conditions are equivalent:
\begin{enumerate}
\item[\textup{(1)}]
$P$ is $w$-faithfully projective.

\item[\textup{(2)}]
There exists $n\ge 1$ such that $L(P)^{\oplus n}$
has a $w$-direct summand isomorphic to $R$.
Equivalently, there exist morphisms in $(R,w)\text{-}\mathrm{mod}$
\[
e:\ L(P)^{\oplus n}\longrightarrow R,
\qquad
s:\ R\longrightarrow L(P)^{\oplus n}
\]
such that $e\circ s=1_R$.

\item[\textup{(3)}]
There exists an infinite index set $I$ such that
\[
\bigoplus_{i\in I} L(P)
\cong
R^{(J)}
\quad\text{in }(R,w)\text{-}\mathrm{mod}
\]
for some index set $J$.
\end{enumerate}

Moreover, if $R$ is $w$-semilocal
(i.e.\ $|\,w\text{-}\Max(R)\,|<\infty$),
then in \textup{(2)} one may take $n=1$.
\end{proposition}

\begin{proof}

\smallskip
\noindent
(1)$\Rightarrow$(2).
Assume that $P$ is $w$-faithfully projective.
By Theorem~\ref{thm:w-Ishi-22},
\[
\hom_w(P,R/\mathfrak m)\neq 0
\qquad
\text{for every }
\mathfrak m\in w\text{-}\Max(R).
\]

We first show that for each
$\mathfrak m\in w\text{-}\Max(R)$
there exists a homomorphism
$f:L(P)\to R$
such that
$f(L(P))\not\subseteq\mathfrak m$.

Suppose, to the contrary, that for some
$\mathfrak m$ we have
$f(L(P))\subseteq\mathfrak m$
for all
$f\in\Hom_R(L(P),R)
=\hom_w(P,R)$.
Since $L(P)$ is projective in
$(R,w)\text{-}\mathrm{mod}$
(Proposition~\ref{prop:w-faith-projective-local}),
for every
\[
g\in
\hom_w(P,R/\mathfrak m)
=
\hom_w\big(L(P),(R/\mathfrak m)_w\big)
\]
there exists
$f_0:L(P)\to R$
such that the diagram
\[
\xymatrix@C=1.5cm{
& L(P)\ar[d]^{g}\ar@{-->}[ld]_{f_0} & \\
R\ar[r]_{L(\pi)} & (R/\mathfrak m)_w\ar[r] & 0
}
\]
commutes,
where $L(\pi)$ is induced by the canonical projection
$\pi:R\twoheadrightarrow R/\mathfrak m$.
Thus $g=L(\pi)\circ f_0$.
But by assumption
$f_0(L(P))\subseteq\mathfrak m$,
so $g=0$,
contradicting
$\hom_w(P,R/\mathfrak m)\neq 0$.

Hence, for each
$\mathfrak m$,
there exists
$f_{\mathfrak m}\in\hom_w(P,R)$
such that
$f_{\mathfrak m}(L(P))\not\subseteq\mathfrak m$.

Consider the $w$-trace ideal
\[
\tr_w(P)
:=
\sum_{f\in\hom_w(P,R)}
f\big(L(P)\big)
\subseteq R.
\]
The above argument shows that
$\tr_w(P)$
is not contained in any
$\mathfrak m\in w\text{-}\Max(R)$.
Therefore
\[
(\tr_w(P))_w=R.
\]

Since the torsion theory $\tau_w$ is of finite type,
there exist
$f_1,\dots,f_n\in\hom_w(P,R)$
such that
\[
\Big(
\sum_{k=1}^n f_k(L(P))
\Big)_w
=
R.
\]

Define
\[
e:L(P)^{\oplus n}\to R,
\qquad
e(x_1,\dots,x_n)
=
\sum_{k=1}^n f_k(x_k).
\]
Then
\[
\coker(e)
=
R\Big/
\sum_{k=1}^n f_k(L(P))
\]
is $\GV$-torsion,
so $e$ is an epimorphism in
$(R,w)\text{-}\mathrm{mod}$.
Since $R$ is projective in
$(R,w)\text{-}\mathrm{mod}$
(Corollary~\ref{cor:projective-obj}),
$e$ splits.
Hence there exists
$s:R\to L(P)^{\oplus n}$
such that
$e\circ s=1_R$.
Thus $R$ is a $w$-direct summand of
$L(P)^{\oplus n}$.

\medskip
\noindent
(2)$\Rightarrow$(3).
Assume
\[
L(P)^{\oplus n}\cong R\oplus Q.
\]
Since $(R,w)\text{-}\mathrm{mod}$
is a Grothendieck category,
it admits arbitrary direct sums.
Taking countable direct sums yields
\[
\big(L(P)^{\oplus n}\big)^{(\mathbb N)}
\cong
R^{(\mathbb N)}
\oplus
Q^{(\mathbb N)}.
\]
But
\[
\big(L(P)^{\oplus n}\big)^{(\mathbb N)}
\cong
L(P)^{(\mathbb N)}.
\]
By the Eilenberg swindle,
\[
L(P)^{(\mathbb N)}
\cong
R^{(\mathbb N)}.
\]
%Thus $\bigoplus_{i\in\mathbb N}L(P)$ is $w$-free.

\medskip
\noindent
(3)$\Rightarrow$(1).
If
\[
L(P)^{(I)}\cong R^{(J)}
\quad
\text{in }
(R,w)\text{-}\mathrm{mod},
\]
then for every
$X\in (R,w)\text{-}\mathrm{mod}$,
\[
\Hom_R(L(P),X)^{I}
\cong
\Hom_R(L(P)^{(I)},X)
\cong
\Hom_R(R^{(J)},X)
\cong
X^{J}.
\]
Hence
$\hom_w(P,X)=0$
implies
$X=0$.
By Theorem~\ref{thm:w-Ishi-22}(2),
$P$ is $w$-faithfully projective.

\medskip
\noindent
\textbf{Semilocal refinement.}

If $R$ is $w$-semilocal with
$w$-maximal ideals
$\mathfrak m_1,\dots,\mathfrak m_t$,
choose
$f_i\in\hom_w(P,R)$
with
$f_i(L(P))\not\subseteq\mathfrak m_i$
for each $i$.
Set
\[
f=\sum_{i=1}^t f_i:L(P)\to R.
\]
Then the image of $f$
has $w$-closure equal to $R$,
so $f$ is an epimorphism in
$(R,w)\text{-}\mathrm{mod}$.
Since $R$ is projective,
$f$ splits.
Thus $R$ is a $w$-direct summand of
$L(P)$,
and one may take $n=1$.
\end{proof}

\begin{proposition}
\label{prop:w-Ishikawa34-proj}
Let $R$ be a commutative ring.
\begin{enumerate}
\item[\textup{(1)}]
If every nonzero $w$-projective $R$-module of finite type
is $w$-faithfully projective, then $R$ is a DW-ring.

\item[\textup{(2)}]
Conversely, if $R$ is indecomposable, then every nonzero
$\GV$-torsion-free $w$-projective $R$-module of finite type
is $w$-faithfully projective.
\end{enumerate}
\end{proposition}

\begin{proof}

\noindent
(1)
If $R$ is not a DW-ring, then by
Example~\ref{exa:non-w-faith-pro}
there exists a nonzero finitely presented
$w$-projective $R$-module that is not
$w$-faithfully projective.
This contradicts the hypothesis.
Hence $R$ must be a DW-ring.

\medskip
\noindent
(2)
We first prove that every nonzero finitely generated
$\GV$-torsion-free $w$-projective $R$-module
is $w$-faithfully projective.

Let $P$ be such a module.
For $\mathfrak m\in w\text{-}\Max(R)$,
the proof of Proposition~\ref{prop:w-faith-projective-local}
shows that
\[
\Hom_R\big(P_w,(R/\mathfrak m)_w\big)=0
\quad\Longleftrightarrow\quad
P_{\mathfrak m}=0.
\]

Introduce the following notation:
\begin{align*}
V_w(I)
&=\{\mathfrak m\in w\text{-}\Max(R)\mid I\subseteq\mathfrak m\},
\\[4pt]
S_w(P)
&=\{\mathfrak m\in w\text{-}\Max(R)\mid P_{\mathfrak m}=0\},
\\[4pt]
T_w(P)
&=\{\mathfrak m\in w\text{-}\Max(R)\mid P_{\mathfrak m}\neq 0\}
=w\text{-}\Max(R)\setminus S_w(P).
\end{align*}

Assume that $P$ is not $w$-faithfully projective.
By Theorem~\ref{thm:w-Ishi-22},
$S_w(P)\neq\varnothing$.
If $T_w(P)=\varnothing$, then
$P_{\mathfrak m}=0$ for all
$\mathfrak m\in w\text{-}\Max(R)$,
so $P$ is $\GV$-torsion.
Since $P$ is also $\GV$-torsion-free,
this forces $P=0$, a contradiction.
Hence $T_w(P)\neq\varnothing$.

For each $\mathfrak m\in T_w(P)$,
there exists
$f_{\mathfrak m}:P_w\to R$
such that
$f_{\mathfrak m}(P_w)\not\subseteq\mathfrak m$
(see the proof of Proposition~\ref{prop:w-Ishikawa24}).
Let $g_{\mathfrak m}=f_{\mathfrak m}|_P$.
Then $g_{\mathfrak m}(P)\not\subseteq\mathfrak m$.

Indeed, choose $x\in P_w$
with $f_{\mathfrak m}(x)\notin\mathfrak m$.
There exists a $\GV$-ideal $J$ with $Jx\subseteq P$.
Since $J\nsubseteq\mathfrak m$,
pick $a\in J\setminus\mathfrak m$.
Set $y=ax\in P$.
Then $g_{\mathfrak m}(y)=f_{\mathfrak m}(y)\notin\mathfrak m$.

Let $I$ be the ideal generated by all such
$g_{\mathfrak m}(y)$.
Then
\[
I\nsubseteq\mathfrak m
\quad
\text{for all }
\mathfrak m\in T_w(P),
\]
so
\[
V_w(I)\subseteq S_w(P).
\]

Conversely, if
$\mathfrak n\in S_w(P)$,
then $P_{\mathfrak n}=0$.
For any $g:P\to R$,
the composition
\[
P\xrightarrow{g}R\to R/\mathfrak n
\hookrightarrow (R/\mathfrak n)_w
\]
extends uniquely to a map
$P_w\to (R/\mathfrak n)_w$
(Proposition~\ref{prop:envelope-basic}(4)).
But
$\Hom_R\big(P_w,(R/\mathfrak n)_w\big)=0$,
so this composition is zero.
Thus $g(P)\subseteq\mathfrak n$,
and hence $I\subseteq\mathfrak n$.
Therefore
\[
V_w(I)=S_w(P).
\]

For each
$\mathfrak n\in S_w(P)$,
since $P_{\mathfrak n}=0$
and $P$ is finitely generated,
there exists
$s_{\mathfrak n}\in R\setminus\mathfrak n$
such that
$s_{\mathfrak n}P=0$.
Hence
$s_{\mathfrak n}I=0$.

Let $J$ be the ideal generated by all such
$s_{\mathfrak n}$.
Then
\[
JI=0,
\qquad
(J+I)_w=R.
\]
It follows that
\[
J_w I_w=0,
\qquad
(J_w+I_w)_w=R.
\]
By \cite[Lemma~5.10]{QZ22},
\[
R\cong J_w\oplus I_w.
\]
This contradicts the indecomposability of $R$.
Hence $P$ is $w$-faithfully projective.

\medskip

Now let $P'$ be a nonzero $\GV$-torsion-free
$w$-projective $R$-module of finite type.
By \cite[Proposition~6.3.13(3)]{WK24},
there exists a finitely generated submodule
$P\subseteq P'$
such that
$P_w=(P')_w$.
Then by \cite[Proposition~6.7.2(1)]{WK24},
$P$ is also $w$-projective.
By the first part of the proof,
$P$ is $w$-faithfully projective.
Finally, by Remark~\ref{rem:w-faith-pro-closed-w-iso},
$P'$ is $w$-faithfully projective as well.
\end{proof}

\section{$w$-faithfully flat modules}

It follows from \cite[Example~3.4]{zk26} that the tensor product of two
$R$-modules in $(R,w)\text{-}\mathrm{mod}$ need not belong to
$(R,w)\text{-}\mathrm{mod}$, even over P$v$MDs.
To remedy this, we introduce the notion of the $w$-tensor product.

\begin{definition}
Let $R$ be a commutative ring.
For $M,N\in\Mod(R)$ define
\[
M\otimes_R^{\,w}N:=L(M\otimes_R N).
\]
This is called the \emph{$w$-tensor product} of $M$ and $N$.
\end{definition}

This construction yields a bifunctor
\[
-\otimes_R^{\,w}-:\Mod(R)\times\Mod(R)
\longrightarrow
(R,w)\text{-}\mathrm{mod},
\]
which may be viewed as the $w$-analogue of Ishikawa's tensor product.

\medskip

\begin{remark}
\label{rem:standing-flat}
Let $R$ be a commutative ring and fix an $R$-module $M$.
Denote the functor $(-)\otimes_R^{\,w}M$ by $U_w(-)$.
We obtain a right exact functor
\[
U_w(-):
\Mod(R)\xrightarrow{-\otimes_R M}
\Mod(R)\xrightarrow{L(-)}
(R,w)\text{-}\mathrm{mod},
\]
which annihilates $\mathcal T(R)$.

The functor
\[
U_w(-)=(-)\otimes_R^{\,w}M:
\Mod(R)\longrightarrow (R,w)\text{-}\mathrm{mod}
\]
is \emph{exact} if and only if $M$ is $w$-flat.
This characterization will be used throughout this section
(cf.~\cite[Theorem~3.9]{QW15}).
\end{remark}

\begin{definition}
\label{def:w-faith-flat}
Let $R$ be a commutative ring.
An $R$-module $M$ is called \emph{$w$-faithfully flat}
if the functor
\[
U_w(-)=(-)\otimes_R^{\,w}M:
\Mod(R)\longrightarrow (R,w)\text{-}\mathrm{mod}
\]
is $w$-faithfully exact.
\end{definition}

It is clear that every $w$-faithfully flat module is $w$-flat.

\medskip

Recall from \cite{zkzh25} that an $R$-module $E$ is called
\emph{$w$-universal injective}
if $E$ is injective and,
for every nonzero $\GV$-torsion-free $R$-module $A$,
there exists an $R$-homomorphism
$f:A\to E$ with $f\neq 0$.
Equivalently,
\[
\Hom_R(A,E)\neq 0
\quad
\text{for every nonzero $\GV$-torsion-free module }A.
\]

\begin{proposition}
\label{prop:w-faith-local}
Let $R$ be a commutative ring.
An $R$-module $M$ is $w$-faithfully flat
if and only if $M_{\mathfrak m}$
is faithfully flat over $R_{\mathfrak m}$
for every
$\mathfrak m\in w\text{-}\Max(R)$.

Equivalently (by the $w$-version of the Ishikawa--Lambek theorem,
cf.~\cite[Proposition~3.6]{I64}),
$M$ is $w$-faithfully flat if and only if
$\Hom_R(M,E)$ is $w$-universal injective
for some (equivalently, for every) injective $R$-module $E$.
\end{proposition}

\begin{proof}
Let
\[
U_w(-)=(-)\otimes_R^{w}M:
\Mod(R)\longrightarrow (R,w)\text{-}\mathrm{mod},
\qquad
X\otimes_R^{w}M=L(X\otimes_R M).
\]

\medskip
\noindent
\emph{Step 1: Exactness and local flatness.}

Exactness of $U_w(-)$ is precisely the $w$-flatness of $M$.
By the local criterion for $w$-flatness,
this is equivalent to
$M_{\mathfrak m}$ being flat over $R_{\mathfrak m}$
for every
$\mathfrak m\in w\text{-}\Max(R)$.

\medskip
\noindent
\emph{Step 2: Faithfulness and local faithful flatness.}

Assume that $U_w(-)$ is exact.
By Theorem~\ref{thm:w-Ishi-11},
$U_w(-)$ is $w$-faithfully exact if and only if
\[
U_w(R/\mathfrak m)\neq 0
\quad
\text{for all }
\mathfrak m\in w\text{-}\Max(R).
\]

Now
\[
U_w(R/\mathfrak m)
=
(R/\mathfrak m)\otimes_R^{w}M
=
L\big((R/\mathfrak m)\otimes_R M\big).
\]

Since $L(-)$ annihilates precisely $\GV$-torsion
and $R/\mathfrak m$ is supported only at $\mathfrak m$,
the above module is nonzero if and only if
\[
\big((R/\mathfrak m)\otimes_R M\big)_{\mathfrak m}\neq 0.
\]

Localizing gives
\[
(R_{\mathfrak m}/\mathfrak mR_{\mathfrak m})
\otimes_{R_{\mathfrak m}}
M_{\mathfrak m}
\neq 0.
\]

Over the local ring $R_{\mathfrak m}$,
this condition, together with flatness of
$M_{\mathfrak m}$,
is equivalent to faithful flatness of
$M_{\mathfrak m}$.

\medskip
\noindent
\emph{Step 3: Ishikawa--Lambek equivalence.}

The equivalence with the $w$-universal injectivity of
$\Hom_R(M,E)$ follows from the
$w$-version of the Ishikawa--Lambek theorem.
\end{proof}

\begin{corollary}
If $f:M\to M'$ is a $w$-isomorphism of $R$-modules,
then $M$ is $w$-faithfully flat if and only if $M'$ is $w$-faithfully flat.
\end{corollary}

\begin{proof}
Since $L(M)\cong L(M')$,
the functors
\[
(-)\otimes_R^{\,w}M
\quad\text{and}\quad
(-)\otimes_R^{\,w}M'
\]
are naturally isomorphic.
Hence one is $w$-faithfully exact if and only if the other is.
\end{proof}

\begin{theorem}\textup{($w$-analogue of Ishikawa's \cite[Theorem~2.1]{I64},
with a maximal-$w$-ideal clause)}
\label{thm:w-def21-plus-full}
Let $M$ be a $w$-flat $R$-module.
Then the following conditions are equivalent.
\begin{enumerate}[label=\textup{(\alph*)}]
\item[\textup{(1)}]
$M$ is $w$-faithfully flat.

\item[\textup{(2)}]
$A\otimes_R^{\,w}M\neq 0$
for every nonzero $w$-module $A$.

\item[\textup{(3)}]
$\varphi\otimes_R^{\,w}1_M\neq 0$
for every nonzero morphism
$\varphi$ in $(R,w)\text{-}\mathrm{mod}$.

\item[\textup{(4)}]
$B\otimes_R^{\,w}M\neq 0$
for every nonzero $\GV$-torsion-free $R$-module $B$.

\item[\textup{(5)}]
For every proper $w$-ideal $I$ of $R$,
\[
L(IM)\neq L(M),
\qquad
\text{equivalently}
\qquad
(R/I)\otimes_R^{\,w}M
\cong
L(M/IM)
\neq 0.
\]

\item[\textup{(6)}]
$(R/\mathfrak m)\otimes_R^{\,w}M\neq 0$
for every $\mathfrak m\in w\text{-}\Max(R)$.

\item[\textup{($6'$)}]
For every $\mathfrak m\in w\text{-}\Max(R)$,
\[
L(M/\mathfrak m M)\neq 0,
\qquad
\text{equivalently}
\qquad
M_{\mathfrak m}/\mathfrak m M_{\mathfrak m}\neq 0.
\]

\item[\textup{(7)}]
$M_{\mathfrak m}$ is faithfully flat over $R_{\mathfrak m}$
for every $\mathfrak m\in w\text{-}\Max(R)$.
\end{enumerate}
\end{theorem}

\begin{proof}
By Theorem~\ref{thm:w-Ishi-11},
conditions (1)--(6) are equivalent,
since $M$ is $w$-flat and hence
$U_w(-)=(-)\otimes_R^{\,w}M$
is exact.

The equivalence of (6), $(6')$, and (7)
follows from Proposition~\ref{prop:w-faith-local}.
\end{proof}

\begin{lemma}\label{lem:Xing-21-plus}
Let $M$ be a $\GV$-torsion-free $R$-module
and let $\mathfrak m$ be a maximal $w$-ideal of $R$.
Then the following conditions are equivalent.
\begin{enumerate}
\item[\textup{(1)}]
$(M_w/(\mathfrak m M)_w)_w=0$.

\item[\textup{(2)}]
$M_w/(\mathfrak m M)_w$ is $\GV$-torsion.

\item[\textup{(3)}]
$M_w=(\mathfrak m M)_w$.

\item[\textup{(4)}]
$M/\mathfrak m M$ is $\GV$-torsion.

\item[\textup{(5)}]
$L(M)=L(\mathfrak m M)$.
\end{enumerate}
\end{lemma}

\begin{proof}
The equivalence of (1)--(4)
is proved in \cite[Lemma~2.1]{X25}.

\smallskip

\noindent
(3) $\Leftrightarrow$ (5).
This follows immediately from the identity
$L(X)=X_w$ for every $\GV$-torsion-free module $X$.
\end{proof}

\begin{remark}
In Theorem~\ref{thm:w-def21-plus-full},
the cyclic nonvanishing condition
\[
L(\mathfrak m M)\neq L(M)
\qquad
(\mathfrak m\in w\text{-}\Max(R))
\]
is, by Lemma~\ref{lem:Xing-21-plus},
equivalent to
\[
(M_w/(\mathfrak m M)_w)_w\neq 0
\quad
\text{for all }
\mathfrak m\in w\text{-}\Max(R).
\]

Moreover,
\cite[Lemma~2.1]{X25}
shows that the vanishing condition
$(M_w/(\mathfrak m M)_w)_w=0$
is equivalent to each of the conditions
\textup{(2)}--\textup{(4)}
in Lemma~\ref{lem:Xing-21-plus}.
Thus our nonvanishing criterion
$L(\mathfrak m M)\neq L(M)$
is precisely the negation of Xing's vanishing condition
\cite{X25}.

Consequently,
the present definition of a
\emph{$w$-faithfully flat module}
is fully compatible with Xing's formulation.
\end{remark}

\begin{proposition}\textup{($w$-analogue of Ishikawa's \cite[Proposition~2.3]{I64})}
\label{prop:w-Ishikawa23}
If $P$ is a $w$-faithfully projective $R$-module, then $P$ is
$w$-projective and $w$-faithfully flat.
Conversely, if $P$ is of finite type, then the reverse implication holds.
\end{proposition}

\begin{proof}
Assume that $P$ is $w$-faithfully projective; that is,
the functor
$H_w(-)=\hom_w(P,-)$
is $w$-faithfully exact.

By Proposition~\ref{prop:exactness of H_w},
the exactness of $H_w(-)$ implies that $P$ is $w$-projective.
Hence $P$ is also $w$-flat
by \cite[Proposition~2.4]{WK14}.

Now let $B$ be a $\GV$-torsion-free $R$-module such that
$B\otimes_R^{\,w} P=0.$
Then
$L(B\otimes_R P)=0,$
so $B\otimes_R P$ is $\GV$-torsion.
Since $P$ and $L(P)$ are $w$-isomorphic,
$B\otimes_R L(P)$ is also $\GV$-torsion.

Let $C$ be a $w$-module.
Then $\Hom_R(B,C)$ is again a $w$-module, and we have
\[
\hom_w\!\big(P,\Hom_R(B,C)\big)
=
\Hom_R\big(L(P),\Hom_R(B,C)\big)
\cong
\Hom_R\big(B\otimes_R L(P),C\big)
=0.
\]
Since $\hom_w(P,-)$ is $w$-faithfully exact,
Theorem~\ref{thm:w-Ishi-11}(2) yields
\[
\Hom_R(B,C)=0
\quad
\text{for every nonzero $w$-module $C$.}
\]

Now let $D$ be any $\GV$-torsion-free module.
Since $D\hookrightarrow D_w$ is essential and $D_w$ is a $w$-module,
the exactness of
\[
0\longrightarrow\Hom_R(B,D)
\longrightarrow
\Hom_R(B,D_w)
\]
implies that $\Hom_R(B,D)=0$.
Thus $\Hom_R(B,D)=0$ for every $\GV$-torsion-free module $D$.

It follows that $B$ is $\GV$-torsion.
But $B$ is also $\GV$-torsion-free,
hence $B=0$.
Therefore, by Theorem~\ref{thm:w-def21-plus-full}(4),
$P$ is $w$-faithfully flat.

\medskip

Conversely, assume that $P$ is of finite type and both
$w$-projective and $w$-faithfully flat.
By \cite[Theorem~6.7.29(8)]{WK24},
$P$ is $w$-flat of finitely presented type.
Then Theorems~\ref{thm:w-def21-plus-full}(7)
and \ref{thm:w-Ishi-22}(7)
imply that $P$ is $w$-faithfully projective.
\end{proof}

\begin{proposition}
\label{prop:w-Ishikawa34-flat-correct}
Let $R$ be a commutative ring.
\begin{enumerate}[label=\textup{(\alph*)}]
\item[\textup{(1)}]
If every nonzero $w$-flat $R$-module is $w$-faithfully flat,
then $R$ is a DW-ring.

\item[\textup{(2)}]
Conversely, if $R$ is indecomposable,
then every nonzero $\GV$-torsion-free $w$-flat
$R$-module of finitely presented type
is $w$-faithfully flat.
\end{enumerate}
\end{proposition}

\begin{proof}
(1)
Suppose that $R$ is not a DW-ring.
Then there exists $J\in\GV(R)$ with $J\neq R$.
Let $M:=R/J$.
Then $M$ is a nonzero $\GV$-torsion module,
hence $w$-flat.

By hypothesis, $M$ is $w$-faithfully flat.
Proposition~\ref{prop:w-faith-local} then implies that
$M_{\mathfrak m}$ is faithfully flat over
$R_{\mathfrak m}$
for all
$\mathfrak m\in w\text{-}\Max(R)$.
But $M$ is $\GV$-torsion,
so $M_{\mathfrak m}=0$ for all such $\mathfrak m$,
a contradiction.
Hence $R$ must be a DW-ring.

\medskip

(2)
Assume that $R$ is indecomposable
and let $M$ be a nonzero $\GV$-torsion-free
$w$-flat $R$-module of finitely presented type.
By \cite[Theorem~6.7.29]{WK24},
$M$ is $w$-projective.
Then Proposition~\ref{prop:w-Ishikawa34-proj}(2)
implies that $M$ is $w$-faithfully projective.
Finally, by Proposition~\ref{prop:w-Ishikawa23},
$M$ is $w$-faithfully flat.
\end{proof}

\section{$w$-faithfully injective modules}

For a fixed $R$-module $E$, define a contravariant functor
\[
V_w(X)
:=
\hom_w(X,E)
:=
\Hom_R\!\big(L(X),\,L(E)\big),
\qquad
X\in \Mod(R).
\]
Thus we obtain a left exact functor
\[
V_w(-):
\Mod(R)
\xrightarrow{L(-)}
(R,w)\text{-}\mathrm{mod}
\xrightarrow{\Hom_R(-,L(E))}
(R,w)\text{-}\mathrm{mod},
\]
which annihilates $\mathcal{T}(R)$.

Recall from \cite{WK14} that an $R$-module $Q$ is called
\emph{$w$-injective} if, for every $w$-exact sequence
\[
0\longrightarrow A\longrightarrow B\longrightarrow C\longrightarrow 0,
\]
the induced sequence
\[
0
\longrightarrow
\Hom_R(C,L(Q))
\longrightarrow
\Hom_R(B,L(Q))
\longrightarrow
\Hom_R(A,L(Q))
\longrightarrow
0
\]
is $w$-exact.
It is shown in \cite[Theorem~4.5]{QW15} that the functor
\[
V_w(-)=\hom_w(-,E):
\Mod(R)\longrightarrow (R,w)\text{-}\mathrm{mod}
\]
is exact if and only if $E$ is $w$-injective.

Recall also from \cite{Wu24} that an $R$-module $N$ is called
\emph{$iw$-split} if there exists a $w$-split short exact sequence
\[
0 \longrightarrow N \longrightarrow E \longrightarrow C \longrightarrow 0
\]
with $E$ injective.

\begin{proposition}
\label{prop:injective-obj}
Let $R$ be a commutative ring and let $E$ be a $w$-module over $R$, i.e.,
$E\in (R,w)\text{-}\mathrm{mod}$.
Then the following conditions are equivalent.
\begin{enumerate}
\item[\textup{(1)}]
$E$ is an injective object of $(R,w)\text{-}\mathrm{mod}$.

\item[\textup{(2)}]
For every morphism $f:A\to B$ in $\Mod(R)$ with
$\ker(f)\in\mathcal T(R)$,
the induced map
\[
\Hom_R\big(L(B),E\big)
\longrightarrow
\Hom_R\big(L(A),E)
\]
is surjective in $(R,w)\text{-}\mathrm{mod}$.

\item[\textup{(3)}]
The functor $V_w(-)$ is exact.

\item[\textup{(4)}]
$E$ is a $w$-injective $R$-module.

\item[\textup{(5)}]
$E$ is an $iw$-split $R$-module.

\item[\textup{(6)}]
$E$ is an injective $R$-module, i.e.,
an injective object of $\Mod(R)$.
\end{enumerate}
\end{proposition}

\begin{proof}

(1)$\Rightarrow$(2).
Let $f:A\to B$ be a morphism in $\Mod(R)$ with
$\ker(f)\in\mathcal T(R)$.
Then
$L\big(\ker(f)\big)=0,$
so the sequence
\[
0 \longrightarrow L(A) \longrightarrow L(B)
\]
is exact in $(R,w)\text{-}\mathrm{mod}$.
Since $E$ is injective in $(R,w)\text{-}\mathrm{mod}$,
the induced map
\[
\Hom_R\big(L(B),E\big)
\longrightarrow
\Hom_R\big(L(A),E)
\]
is surjective.

\medskip

(2)$\Rightarrow$(3).
This follows immediately from the identity
\[
V_w(-)=\Hom_R\big(L(-),E\big).
\]

\medskip

(3)$\Rightarrow$(1).
Let
\[
0\longrightarrow A\longrightarrow B
\]
be exact in $(R,w)\text{-}\mathrm{mod}$.
Since $L(A)=A$ and $L(B)=B$,
exactness of $V_w(-)$ yields
\[
\Hom_R(B,E)
\longrightarrow
\Hom_R(A,E)
\longrightarrow
0
\]
exact in $(R,w)\text{-}\mathrm{mod}$.
Thus $E$ is injective in $(R,w)\text{-}\mathrm{mod}$.

\medskip

(3)$\Leftrightarrow$(4).
This is precisely \cite[Theorem~4.5]{QW15}.

\medskip

(4)$\Leftrightarrow$(5).
This is proved in \cite[Corollary~2.5]{Wu24}.

\medskip

(1)$\Rightarrow$(6).
This follows from \cite[Chapter~X, Proposition~1.7]{S75},
since $(R,w)\text{-}\mathrm{mod}$ is a reflective subcategory of
$\Mod(R)$ with exact reflector.

\medskip

(6)$\Rightarrow$(4).
See \cite[Corollary~3.5]{WK14}.
\end{proof}

\begin{corollary}\label{exm:injective-obj}
Every $\GV$-torsion-free injective $R$-module is an injective object of
$(R,w)\text{-}\mathrm{mod}$.
\end{corollary}

\begin{proof}
If $E$ is injective and $\GV$-torsion-free, then $L(E)=E$.
By Proposition~\ref{prop:injective-obj}, injective $R$-modules that are
$w$-modules are injective objects in $(R,w)\text{-}\mathrm{mod}$.
\end{proof}

\begin{corollary}\label{cor:w-injective}
Let $E$ be an $R$-module. The following conditions are equivalent.
\begin{enumerate}
\item[\textup{(1)}]
$E$ is $w$-injective.

\item[\textup{(2)}]
$L(E)$ is $w$-injective.

\item[\textup{(3)}]
$L(E)$ is an $iw$-split $R$-module.

\item[\textup{(4)}]
$L(E)$ is an injective $R$-module.
\end{enumerate}
\end{corollary}

\begin{proof}
Since $E$ and $L(E)$ are $w$-isomorphic and $L(E)$ is a $w$-module,
$E$ is $w$-injective if and only if $L(E)$ is $w$-injective
(see \cite[Example~3.2]{WK14}).
The remaining equivalences follow from
Proposition~\ref{prop:injective-obj}.
\end{proof}

\begin{definition}
\label{def:w-faith-inj}
Let $R$ be a commutative ring and let $E$ be an $R$-module.
We say that $E$ is \emph{$w$-faithfully injective} if the functor
\[
V_w(-)=\hom_w(-,E):
\Mod(R)\longrightarrow (R,w)\text{-}\mathrm{mod}
\]
is $w$-faithfully exact.
Equivalently, the object $L(E)$ is faithfully injective in the abelian
category $(R,w)\text{-}\mathrm{mod}$.
\end{definition}

It is immediate that every $w$-faithfully injective module is
$w$-injective.

\begin{remark}\label{rem:w-faith-injective-closed-w-iso}
If $E$ and $E'$ are $w$-isomorphic $R$-modules, then
$L(E)\cong L(E')$.
Hence $E$ is $w$-faithfully injective if and only if so is $E'$.
\end{remark}

\begin{theorem}\textup{($w$-version of Ishikawa's \cite[Theorem~3.1]{I64})}
\label{thm:w-ishi-31}
Let $E$ be a $w$-injective $R$-module.
Then the following conditions are equivalent.
\begin{enumerate}
\item[\textup{(1)}]
$E$ is $w$-faithfully injective.

\item[\textup{(2)}]
$\hom_w(A,E)\neq 0$ for every nonzero $w$-module $A$ over $R$.

\item[\textup{(3)}]
For every nonzero morphism $\varphi:A\to B$ in
$(R,w)\text{-}\mathrm{mod}$,  the induced map
\[
\Hom_R(A,L(E))\xrightarrow{\ \varphi \circ  -  }
\Hom_R(B,L(E))
\]
is nonzero; equivalently, there exists
$\psi\in\Hom_R(B,L(E))$ such that
$\psi\circ\varphi\neq 0$.

\item[\textup{(4)}]
$\hom_w(B,E)\neq 0$ for every nonzero
$\GV$-torsion-free $R$-module $B$.

\item[\textup{(5)}]
$\hom_w(R/I,E)\neq 0$ for every proper $w$-ideal $I$ of $R$.

\item[\textup{(6)}]
$\hom_w(R/\fkm,E)\neq 0$ for every
$\fkm\in w\text{\rm-}\Max(R)$.
\end{enumerate}
\end{theorem}

\begin{proof}
Since $E$ is $w$-injective, the functor
\[
G=V_w(-)=\hom_w(-,E)
\]
is exact and annihilates $\mathcal T(R)$.
Applying Theorem~\ref{thm:w-Ishi-11} to $G$
yields the desired equivalences.
\end{proof}

\begin{corollary}\label{cor:w-faith-injective}
Let $E$ be an $R$-module.
Then $E$ is $w$-faithfully injective if and only if
$L(E)$ is $w$-universal injective.

Moreover, if $E$ is a $w$-module, then
$E$ is $w$-faithfully injective if and only if it is
$w$-universal injective.
\end{corollary}

\begin{proof}
By Corollary~\ref{cor:w-injective},
$E$ is $w$-faithfully injective
if and only if $L(E)$ is $w$-injective.
By Theorem~\ref{thm:w-ishi-31}(4),
this is equivalent to
$\hom_w(B,E)\neq 0$
for every nonzero $\GV$-torsion-free module $B$.
Since
\[
\Hom_R(A,L(E))\cong\Hom_R(A_w,L(E))
\]
for every $\GV$-torsion-free module $A$,
this is precisely the definition of
$w$-universal injectivity of $L(E)$.
\end{proof}

Let $A,B$ be $R$-modules and let
$c_A:A\to L(A)$ and $c_B:B\to L(B)$ be the canonical homomorphisms.
Define
\[
c_A\oplus c_B:A\oplus B\longrightarrow L(A)\oplus L(B)
\]
by
\[
(c_A\oplus c_B)(a,b)=\big(c_A(a),c_B(b)\big)
\quad\text{for all }(a,b)\in A\oplus B.
\]
Since $L(A)\oplus L(B)$ is a $w$-module, the morphism
$c_A\oplus c_B$ induces a homomorphism
\[
L(c_A\oplus c_B):
L(A\oplus B)\longrightarrow L(A)\oplus L(B).
\]

\begin{lemma}\label{lem:L-direct-sum}
Under the above notation,
\[
L(c_A\oplus c_B):
L(A\oplus B)\longrightarrow L(A)\oplus L(B)
\]
is an $R$-isomorphism.
\end{lemma}

\begin{proof}
	Consider the following commutative diagram with exact rows in $(R,w)\text{-}\mathrm{mod}$:
	\[
	\xymatrix{
            0\ar[r] & L(A)\ar[r]\ar@{=}[d] 
            & L(A\oplus B)\ar[r]\ar[d]^{L(c_A\oplus c_B)} 
            & L(B)\ar[r]\ar@{=}[d] 
            & 0 \\
            0\ar[r] & L(A)\ar[r] 
            & L(A)\oplus L(B)\ar[r] 
            & L(B)\ar[r] 
            & 0.
        }
	\]
	
	Since the outer vertical maps are identities and both rows are exact, it follows that 
	$L(c_A\oplus c_B)$ is an isomorphism in $(R,w)\text{-}\mathrm{mod}$. 
	Consequently, it is also an isomorphism in $\Mod(R)$.
\end{proof}

\begin{lemma}\label{lem:w-injective-closed-direct-sum}
Let $A,B$ be $R$-modules.
Then $A\oplus B$ is $w$-injective if and only if both
$A$ and $B$ are $w$-injective.
\end{lemma}

\begin{proof}
By Lemma~\ref{lem:L-direct-sum},
\[
L(A\oplus B)\cong L(A)\oplus L(B).
\]
By Corollary~\ref{cor:w-injective}, an $R$-module $X$ is $w$-injective
if and only if $L(X)$ is injective in $\Mod(R)$.
Since injectivity is preserved under finite direct sums,
the result follows.
\end{proof}

\begin{corollary}\textup{($w$-analogue of Ishikawa's \cite[Corollary~3.2]{I64})}
\label{cor:w-Ishikawa32}
\begin{enumerate}
\item[\textup{(1)}]
If $E$ is $w$-faithfully injective and $E'$ is $w$-injective,
then $E\oplus E'$ is $w$-faithfully injective.

\item[\textup{(2)}]
Let
\[
Q=\E_R\!\left(\bigoplus_{\mathfrak n\in w\text{\rm-}\Max(R)}R/\mathfrak n\right),
\qquad
Q'=\prod_{\mathfrak n\in w\text{\rm-}\Max(R)} \E_R(R/\mathfrak n),
\]
where $\E_R(-)$ denotes the injective envelope in $\Mod(R)$.
Then both $Q$ and $Q'$ are $w$-faithfully injective.
\end{enumerate}
\end{corollary}

\begin{proof}
\textup{(1)}
By Lemma~\ref{lem:w-injective-closed-direct-sum},
$E\oplus E'$ is $w$-injective.

Let $A$ be a $w$-module such that
$\hom_w(A,E\oplus E')=0.$

Using Lemma~\ref{lem:L-direct-sum},
\begin{align*}
\hom_w(A,E\oplus E')
&=\Hom_R\big(A,L(E\oplus E')\big) \\
&\cong
\Hom_R\big(A,L(E)\oplus L(E')\big) \\
&\cong
\Hom_R\big(A,L(E)\big)
\oplus
\Hom_R\big(A,L(E')\big) \\
&=
\hom_w(A,E)\oplus \hom_w(A,E').
\end{align*}
Hence $\hom_w(A,E)=0$.
Since $E$ is $w$-faithfully injective,
Theorem~\ref{thm:w-ishi-31}(2) implies $A=0$.
Therefore $E\oplus E'$ is $w$-faithfully injective.

\smallskip
\textup{(2)}
Both $Q$ and $Q'$ are injective $R$-modules and
$\GV$-torsion-free.
Hence they are $w$-injective by Corollary~\ref{cor:w-injective}.

For each $\fkm\in w\text{\rm-}\Max(R)$,
\[
\hom_w(R/\fkm,Q)
=
\Hom_R\big((R/\fkm)_w,Q\big)
\cong
\Hom_R(R/\fkm,Q)
\neq 0,
\]
since $R/\fkm$ embeds into its injective envelope.

Similarly,
\[
\hom_w(R/\fkm,Q')
\cong
\prod_{\mathfrak n\in w\text{\rm-}\Max(R)}
\Hom_R(R/\fkm,\E_R(R/\mathfrak n)).
\]
The component corresponding to $\mathfrak n=\fkm$
is nonzero, so the product is nonzero.

Thus, by Theorem~\ref{thm:w-ishi-31},
both $Q$ and $Q'$ are $w$-faithfully injective.
\end{proof}

Recall from \cite[Definition 6.4.1]{WK24} that a nonzero $w$-module $M$
is said to be $w$-simple if $M$ has no nontrivial $w$-submodules.
Recall from \cite{G86} that an $R$-module $M$ is said to be
$\tau_w$-simple if it is not $\tau_w$-torsion and
$\tor_{\GV}(M)$ is the unique proper $\tau_w$-pure submodule of $M$.

\begin{lemma}\label{lem:w-semisimple}
Let $M$ be a $w$-module over $R$ and let $\{T_i\}_{i\in I}$
be a family of $w$-simple $R$-submodules of $M$ such that
\[
M=\sum_{i\in I}T_i.
\]
Then:
\begin{enumerate}
\item[\textup{(1)}]
For each $w$-submodule $K$ of $M$ there exists a subset $J\subseteq I$
such that $\{T_j\}_{j\in J}$ is independent
(i.e.\ $\sum_{j\in J}T_j=\bigoplus_{j\in J}T_j$) and
\[
M
=
K
\oplus
\left(\bigoplus_{j\in J}T_j\right).
\]

\item[\textup{(2)}]
For every $w$-simple $R$-submodule $T$ of $M$,
there exists $i\in I$ such that $T\cong T_i$.
\end{enumerate}
\end{lemma}

\begin{proof}
(1)
Let $K$ be a $w$-submodule of $M$.
Consider
\[
\Gamma
=
\left\{
J\subseteq I
\ \middle|\
\{T_j\}_{j\in J}\text{ is independent and }
K\cap\Big(\sum_{j\in J}T_j\Big)=0
\right\}.
\]
Ordered by inclusion, every chain has an upper bound,
so by Zorn's lemma $\Gamma$ has a maximal element $J$.

Set
\[
N
=
K
+
\sum_{j\in J}T_j.
\]
By construction,
\[
N
=
K
\oplus
\left(\bigoplus_{j\in J}T_j\right).
\]

We claim $N=M$.
Let $i\in I$.
Since $T_i$ is $w$-simple and $T_i\cap N$
is a $w$-submodule of $T_i$, either
\[
T_i\cap N=0
\quad\text{or}\quad
T_i\subseteq N.
\]
If $T_i\cap N=0$, then
$J\cup\{i\}\in\Gamma$,
contradicting maximality.
Hence $T_i\subseteq N$ for all $i$,
so $M=N$.

\smallskip
(2)
Applying (1) with $K=0$,
there exists $J\subseteq I$ such that
\[
M=\bigoplus_{j\in J}T_j.
\]

Now apply (1) with $K=T$.
Then
\[
M
=
T
\oplus
\left(\bigoplus_{l\in L}T_l\right)
\]
for some $L\subseteq J$.

But also
\[
M
=
\left(\bigoplus_{j\in J\setminus L}T_j\right)
\oplus
\left(\bigoplus_{l\in L}T_l\right).
\]
Hence
\[
T
\cong
\bigoplus_{j\in J\setminus L}T_j.
\]
Since $T$ is $w$-simple,
$|J\setminus L|=1$.
Thus $T\cong T_i$ for some $i\in J\subseteq I$.
\end{proof}

\begin{lemma}\label{lem:sum-of-w-simple}
Let $M$ be a $w$-module and
$M_1,\dots,M_n$ be pairwise non-isomorphic
$w$-simple $R$-submodules of $M$.
Then
\[
\sum_{i=1}^n M_i
=
\bigoplus_{i=1}^n M_i.
\]
\end{lemma}

\begin{proof}
We argue by induction on $n$.

The case $n=1$ is trivial.
Assume $n>1$ and the result holds for $n-1$.

Suppose for some $i$,
\[
M_i
\cap
\sum_{j\ne i}M_j
\ne 0.
\]
By induction,
\[
\sum_{j\ne i}M_j
=
\bigoplus_{j\ne i}M_j,
\]
which is a $w$-module.
Hence
\[
M_i
\cap
\sum_{j\ne i}M_j
\]
is a $w$-submodule of $M_i$.
Since $M_i$ is $w$-simple,
\[
M_i
\subseteq
\sum_{j\ne i}M_j.
\]
Thus
\[
M_i
\cong
M_j
\]
for some $j\ne i$,
contradicting the hypothesis.
Therefore all intersections are zero.
\end{proof}

\begin{corollary}\textup{(Unified $w$-analogue of Ishikawa's \cite[Corollary~3.3]{I64})}
\label{cor:w-Ishikawa33}
Let $E$ be a $w$-injective $R$-module and let $Q$
be as in Corollary~\ref{cor:w-Ishikawa32}.
The following are equivalent:
\begin{enumerate}
\item[\textup{(1)}]
$E$ is $w$-faithfully injective.

\item[\textup{(2)}]
$L(E)$ contains an isomorphic copy of every
$\tau_w$-simple $R$-module.

\item[\textup{(3)}]
$L(E)$ has a direct summand isomorphic to $Q$.
\end{enumerate}
\end{corollary}

\begin{proof}
(1)$\Rightarrow$(2)
Let $M$ be a $\tau_w$-simple $R$-module.
By \cite[Corollary 2.4 and Proposition 2.3]{KZL19},
$M_w
\cong
(R/\fkm)_w$
for some $\fkm\in w\text{-}\Max(R)$.

Since $E$ is $w$-faithfully injective,
Theorem~\ref{thm:w-ishi-31}(6) gives
$\hom_w(R/\fkm,E)\ne 0.$
Hence
\[
\Hom_R(M,L(E))
\cong
\Hom_R(M_w,L(E))
\ne 0.
\]

Let $f:M\to L(E)$ be nonzero.
Because $M$ is $\tau_w$-simple and
$\ker(f)$ is $\tau_w$-closed,
$f$ must be injective.
Thus $M$ embeds into $L(E)$.

\smallskip
(2)$\Rightarrow$(3)
For each $\fkm\in w\text{-}\Max(R)$,
$(R/\fkm)_w$ is both $\tau_w$-simple
and $w$-simple.
By (2), there exists a monomorphism
\[
f_\fkm:
(R/\fkm)_w
\to
L(E).
\]

Let
\[
f:
\bigoplus_{\fkm}
(R/\fkm)_w
\to
L(E)
\]
be defined by
\[
f([x_\fkm])
=
\sum_\fkm f_\fkm(x_\fkm).
\]

Since the images are pairwise non-isomorphic
$w$-simple submodules,
Lemma~\ref{lem:sum-of-w-simple}
implies $f$ is injective.

Because $L(E)$ is injective and
$Q$ is the injective envelope of
$\bigoplus_\fkm R/\fkm$,
$f$ extends to a monomorphism
$g:
Q
\to
L(E).$
Since $Q$ is injective,
$g$ splits.
Thus $Q$ is a direct summand of $L(E)$.

\smallskip
(3)$\Rightarrow$(1)
If
$L(E)
=
Q
\oplus
L$
with $L$ injective,
then by Corollary~\ref{cor:w-Ishikawa32},
$L(E)$ is $w$-faithfully injective.
Since $E$ and $L(E)$ are $w$-isomorphic,
$E$ is $w$-faithfully injective.
\end{proof}

\begin{corollary}
Let $E$ be a $w$-injective $w$-module over $R$
(equivalently, $E$ is $\GV$-torsion-free and injective),
and let $Q$ be as in Corollary~\ref{cor:w-Ishikawa32}.
Then the following conditions are equivalent:
\begin{enumerate}
\item[\textup{(1)}]
$E$ is $w$-faithfully injective.

\item[\textup{(2)}]
$E$ contains an isomorphic image of every $\tau_w$-simple $R$-module.

\item[\textup{(3)}]
$E$ has a direct summand isomorphic to $Q$.
\end{enumerate}
\end{corollary}

\begin{proposition}\textup{($w$-analogue of Ishikawa's \cite[Proposition~3.4]{I64})}
\label{prop:w-Ishikawa34}
Let $R$ be a commutative ring.
If every nonzero $w$-injective $R$-module is $w$-faithfully injective,
then $R$ is $w$-local; that is,
\[
|\,w\text{\rm-}\Max(R)\,|=1.
\]
Consequently, $R$ is a DW-ring.
\end{proposition}

\begin{proof}
Assume every nonzero $w$-injective $R$-module is $w$-faithfully injective.

Suppose, toward a contradiction, that
\[
|\,w\text{\rm-}\Max(R)\,|\ge 2,
\]
and choose distinct $w$-maximal ideals
$\mathfrak m\ne\mathfrak n$.

Then $\E_R(R/\mathfrak m)$ is $w$-injective,
hence $w$-faithfully injective by hypothesis.
Therefore,
\[
\hom_w(R/\mathfrak n,\E_R(R/\mathfrak m))\ne 0.
\]
Since $(R/\mathfrak n)_w\cong R/\mathfrak n$,
this gives
\[
\Hom_R(R/\mathfrak n,\E_R(R/\mathfrak m))\ne 0.
\]

Let
\[
f:R/\mathfrak n\to \E_R(R/\mathfrak m)
\]
be nonzero.

Because $R/\mathfrak n$ is $w$-simple and
$\ker(f)$ is a $w$-submodule,
$f$ must be injective.
Thus
\[
R/\mathfrak n
\cong
f(R/\mathfrak n)
\subseteq
\E_R(R/\mathfrak m).
\]

Since $\E_R(R/\mathfrak m)$
is an essential extension of $R/\mathfrak m$,
we must have
\[
f(R/\mathfrak n)\cap R/\mathfrak m\ne 0.
\]
Choose
\[
0\ne x\in f(R/\mathfrak n)\cap R/\mathfrak m.
\]
Then
\[
\ann_R(x)=\mathfrak m
\quad\text{and}\quad
\ann_R(x)=\mathfrak n,
\]
so $\mathfrak m=\mathfrak n$,
a contradiction.

Hence $|\,w\text{\rm-}\Max(R)\,|=1$,
so $R$ is $w$-local.

Finally, by \cite[Proposition 2.6]{KZL19},
a $w$-local ring is a DW-ring.
\end{proof}

A commutative ring $R$ is called \emph{$w$-Noetherian} if every ascending chain of $w$-ideals of $R$ stabilizes. 
Equivalently, every $w$-ideal of $R$ is of finite type; that is, for each $w$-ideal $I$ of $R$, there exists a finitely generated ideal $J \subseteq I$ such that $I = J_w$ (\cite[Theorem 6.8.5]{WK24}).

\begin{proposition}\textup{($w$-analogue of Ishikawa's \cite[Proposition~3.5]{I64})}
\label{prop:w-localization-criterion}
Let $R$ be a $w$-Noetherian commutative ring and let
$E$ be a $w$-injective $R$-module.
Then $E$ is $w$-faithfully injective
if and only if
$E_{\mathfrak m}$ is a faithfully injective
$R_{\mathfrak m}$-module for every
$\mathfrak m\in w\text{\rm-}\Max(R)$.
\end{proposition}

\begin{proof}
($\Rightarrow$)
Assume $E$ is $w$-faithfully injective and
fix $\mathfrak m\in w\text{\rm-}\Max(R)$.

By Corollary~\ref{cor:w-injective},
$L(E)$ is a $\GV$-torsion-free injective $R$-module.
Hence by \cite[Theorem 2.13]{KW13},
\[
E_{\mathfrak m}\cong L(E)_{\mathfrak m}
\]
is injective over $R_{\mathfrak m}$.

Since $E$ is $w$-faithfully injective,
\[
\hom_w(R/\mathfrak m,E)\ne 0.
\]
Thus
\[
\Hom_R(R/\mathfrak m,L(E))\ne 0.
\]

By \cite[Lemma 2.12(1)]{KW13},
for $\mathfrak n\ne\mathfrak m$,
\[
\Hom_R(R/\mathfrak m,L(E))_{\mathfrak n}=0.
\]
Therefore
\[
\Hom_R(R/\mathfrak m,L(E))_{\mathfrak m}\ne 0.
\]

Localizing yields
\[
\Hom_{R_{\mathfrak m}}
\big(R_{\mathfrak m}/\mathfrak mR_{\mathfrak m},
E_{\mathfrak m}\big)
\ne 0.
\]

Over a local ring,
injectivity together with this condition
is equivalent to faithful injectivity.
Hence $E_{\mathfrak m}$ is faithfully injective.

\smallskip
($\Leftarrow$)
Suppose $E_{\mathfrak m}$ is faithfully injective
for every $\mathfrak m\in w\text{\rm-}\Max(R)$.

Let $I$ be a proper $w$-ideal of $R$.
Then there exists
$\mathfrak n\in w\text{\rm-}\Max(R)$
with
$I_{\mathfrak n}\ne R_{\mathfrak n}$.

Faithful injectivity of $E_{\mathfrak n}$
implies
\[
\Hom_{R_{\mathfrak n}}
\big(R_{\mathfrak n}/I_{\mathfrak n},
E_{\mathfrak n}\big)
\ne 0.
\]

But
\[
\hom_w(R/I,E)_{\mathfrak n}
\cong
\Hom_{R_{\mathfrak n}}
\big(R_{\mathfrak n}/I_{\mathfrak n},
E_{\mathfrak n}\big),
\]
so
\[
\hom_w(R/I,E)_{\mathfrak n}\ne 0.
\]

Hence
\[
\hom_w(R/I,E)\ne 0.
\]
By Theorem~\ref{thm:w-ishi-31}(5),
$E$ is $w$-faithfully injective.
\end{proof}

\begin{remark}
The $w$-Noetherian hypothesis in Proposition~\ref{prop:w-localization-criterion}
is used only to ensure that for every proper $w$-ideal $I$ of $R$,
the module $R/I$ is of finitely presented type.
No finiteness assumption on $E$ is required beyond $w$-injectivity.
\end{remark}

\begin{corollary}
\label{cor:support-test-w-UI}
Assume that $R$ is $w$-Noetherian and let $E$ be a $w$-injective
$w$-module over $R$.
Then the following conditions are equivalent.
\begin{enumerate}
\item[\textup{(1)}]
$E$ is $w$-faithfully injective (equivalently, $w$-universal injective).

\item[\textup{(2)}]
Every $w$-maximal ideal lies in the support of $E$, that is,
\[
w\text{\rm-}\Max(R)\subseteq \Supp_R(E),
\]
equivalently,
\[
\Ann_R(E)\subseteq \fkm
\quad
\text{for every }\fkm\in w\text{\rm-}\Max(R).
\]
\end{enumerate}
\end{corollary}

\begin{proof}
Since $E$ is $w$-injective, the contravariant functor
\[
V_w(-):=\hom_w(-,E)
\]
is exact.

\smallskip
\noindent (1)$\Rightarrow$(2).
If $E$ is $w$-faithfully injective, then by
Theorem~\ref{thm:w-ishi-31}(6),
\[
\hom_w(R/\fkm,E)\ne 0
\quad
\text{for every }\fkm\in w\text{\rm-}\Max(R).
\]
Since $(R/\fkm)_w\cong R/\fkm$, this is equivalent to
\[
\Hom_R(R/\fkm,E)\ne 0.
\]
But
\[
\Hom_R(R/\fkm,E)\ne 0
\quad\Longleftrightarrow\quad
\Ann_R(E)\subseteq \fkm,
\]
which is equivalent to $\fkm\in\Supp_R(E)$.
Hence
\[
w\text{\rm-}\Max(R)\subseteq \Supp_R(E).
\]

\smallskip
\noindent (2)$\Rightarrow$(1).
Assume that
\[
\Ann_R(E)\subseteq \fkm
\quad
\text{for every }\fkm\in w\text{\rm-}\Max(R).
\]
Then
\[
\Hom_R(R/\fkm,E)\ne 0
\quad
\text{for all }\fkm\in w\text{\rm-}\Max(R),
\]
and therefore
\[
\hom_w(R/\fkm,E)\ne 0
\quad
\text{for all }\fkm\in w\text{\rm-}\Max(R).
\]
By Theorem~\ref{thm:w-ishi-31},
$E$ is $w$-faithfully injective.
\end{proof}

\begin{proposition}\textup{($w$-analogue of Ishikawa's \cite[Proposition~3.6]{I64})}
\label{prop:w-Ishikawa36}
Let $M$ be an $R$-module and let $E$ be a $w$-faithfully injective
$R$-module. Put
\[
H\ :=\ \hom_w(M,E).
\]
Then:
\begin{enumerate}
\item[\textup{(1)}]
$M$ is $w$-flat if and only if $H$ is $w$-injective, and if and only if
$H$ is an injective $R$-module.

\item[\textup{(2)}]
$M$ is $w$-faithfully flat if and only if $H$ is $w$-faithfully injective.
\end{enumerate}
\end{proposition}

\begin{proof}
We first establish the $w$-tensor–Hom adjunction:
for every $R$-module $X$ there is a natural isomorphism
\begin{equation}\label{eq:w-adjunction-Ish36}\tag{$\dagger$}
\hom_w\!\big(X\otimes_R^{\,w} M,\,E\big)
\ \cong\
\hom_w\!\big(X,\,\hom_w(M,E)\big)
\ =\
\hom_w(X,H).
\end{equation}

Indeed, since
\[
X\otimes_R^{\,w}M=L(X\otimes_R M)
\quad\text{and}\quad
\hom_w(A,B)=\Hom_R(L(A),L(B)),
\]
we compute
\begin{align*}
\hom_w\!\big(X\otimes_R^{\,w}M,E\big)
&=\Hom_R\!\big(L(X\otimes_R M),L(E)\big)\\
&\cong \Hom_R\!\big(X\otimes_R M,L(E)\big)\\
&\cong \Hom_R\!\big(X,\Hom_R(M,L(E))\big)\\
&\cong \Hom_R\!\big(L(X),\Hom_R(M,L(E))\big)\\
&\cong \Hom_R\!\big(L(X),\Hom_R(L(M),L(E))\big)\\
&=\hom_w(X,H).
\end{align*}
Naturality follows from functoriality of $L(-)$ and the classical tensor–Hom adjunction.

\smallskip
\noindent (1)
Assume first that $M$ is $w$-flat.
Then the functor
\[
U(-)\ :=\ -\otimes_R^{\,w}M
\]
is exact.
Since $E$ is $w$-injective, the functor
\[
V(-)=\hom_w(-,E)
\]
is exact.
Hence the composite $V\circ U$ is exact.
By \eqref{eq:w-adjunction-Ish36},
\[
V(U(X))\cong \hom_w(X,H)
\]
for all $X$.
Thus the functor $X\mapsto \hom_w(X,H)$ is exact,
which means that $H$ is $w$-injective.

Conversely, assume that $H$ is $w$-injective.
Then $X\mapsto \hom_w(X,H)$ is exact.
By \eqref{eq:w-adjunction-Ish36},
\[
X\longmapsto \hom_w(X\otimes_R^{\,w}M,E)
\]
is exact.
Since $E$ is $w$-faithfully injective,
the functor $\hom_w(-,E)$ is faithfully exact.
Therefore exactness of
\[
X\longmapsto \hom_w(X\otimes_R^{\,w}M,E)
\]
forces
\[
X\longmapsto X\otimes_R^{\,w}M
\]
to be exact.
Hence $M$ is $w$-flat.

Finally, since $H$ is a $w$-module,
$H$ is $w$-injective if and only if it is injective
as an $R$-module by Proposition~\ref{prop:injective-obj}.

\smallskip
\noindent (2)
Assume that $M$ is $w$-faithfully flat.
Then $U(-)=-\otimes_R^{\,w}M$ is faithfully exact.

Let $0\ne X\in (R,w)\text{-}\mathrm{mod}$.
Then $U(X)\ne 0$.
Since $E$ is $w$-faithfully injective,
\[
\hom_w(U(X),E)\ne 0.
\]
By \eqref{eq:w-adjunction-Ish36},
\[
\hom_w(X,H)\cong \hom_w(U(X),E)\ne 0.
\]
Thus $\hom_w(-,H)$ annihilates no nonzero object.
Together with exactness from part (1),
this shows that $H$ is $w$-faithfully injective.

Conversely, assume that $H$ is $w$-faithfully injective.
Let $0\ne X\in (R,w)\text{-}\mathrm{mod}$.
Then
\[
\hom_w(X,H)\ne 0.
\]
By \eqref{eq:w-adjunction-Ish36},
\[
\hom_w(X\otimes_R^{\,w}M,E)\ne 0.
\]
Since $E$ is $w$-faithfully injective,
$\hom_w(Y,E)\ne 0$ implies $Y\ne 0$.
Hence
\[
X\otimes_R^{\,w}M\ne 0
\quad
\text{for all }0\ne X.
\]
Together with $w$-flatness from part (1),
this shows that $U(-)$ is faithfully exact.
Therefore $M$ is $w$-faithfully flat.
\end{proof}

\begin{theorem}\textup{($w$-analogue of Ishikawa's \cite[Theorem~4.1]{I64})}
\label{thm:w-Ishikawa41}
Let $E$ be a $w$-injective $w$-module.
For each ideal $A$ of $R$, set
\[
A^{*}\ :=\ 0\!:_{E}\!A
\ =\ \{\,q\in E\mid Aq=0\,\},
\qquad
A^{**}\ :=\ 0\!:_{R}\!A^{*}
\ =\ \{\,r\in R\mid rA^{*}=0\,\}.
\]
Then the following conditions are equivalent.
\begin{enumerate}[label=\textup{(\arabic*)}]
\item\label{it:wI41-faith}
$E$ is $w$-faithfully injective.

\item\label{it:wI41-bidual-wideals}
$I=I^{**}$ for every proper $w$-ideal $I$ of $R$.

\item\label{it:wI41-bidual-wmax}
$\fkm=\fkm^{**}$ for every $\fkm\in w\text{\rm-}\Max(R)$.
\end{enumerate}
\end{theorem}

\begin{proof}
Since $E$ is $w$-injective, the contravariant functor
\[
V(-)\ :=\ \hom_w(-,E)
\]
is exact.

\smallskip
\noindent\emph{Step 0: cyclic identifications.}
For every ideal $I\lhd R$ there is a canonical isomorphism
\begin{equation}\label{eq:wI41-Hom-Istar}\tag{$\ddagger$}
\Hom_R(R/I,E)\ \xrightarrow{\ \sim\ }\ I^{*},
\qquad
f\longmapsto f(\overline{1}),
\end{equation}
whose inverse sends $q\in I^{*}$ to the map
$R/I\to E$, $\overline{r}\mapsto rq$.
Hence
\[
I^{**}=0_R:I^{*}
=\{\,r\in R\mid rI^{*}=0\,\}.
\]

If $I$ is a $w$-ideal, then $R/I$ is $\GV$-torsion-free.
Since $E$ is a $w$-module, we obtain
\begin{equation}\label{eq:wI41-homw-cyclic}\tag{$\star$}
\hom_w(R/I,E)
\ \cong\ 
\Hom_R(R/I,E)
\ \cong\ 
I^{*}.
\end{equation}

\smallskip
\noindent
(1)$\Rightarrow$(2)
Let $I$ be a proper $w$-ideal.
The inclusion $I\subseteq I^{**}$ always holds by definition.

Assume $I\subsetneq I^{**}$.
Choose $r\in I^{**}\setminus I$,
so $\overline r\neq 0$ in $R/I$.

Since $E$ is $w$-faithfully injective,
Theorem~\ref{thm:w-ishi-31} (cyclic test) yields
\[
\hom_w(R/I,E)\neq 0.
\]
Hence, by \eqref{eq:wI41-homw-cyclic}, there exists
$q\in I^{*}$ with $q\neq 0$.
Equivalently, there exists
$h\in\Hom_R(R/I,E)$ such that
$q=h(\overline{1})\neq 0$.

Then
\[
h(\overline r)=r h(\overline 1)=rq.
\]
Since $\overline r\neq 0$ and $h\neq 0$,
we may choose $h$ so that $h(\overline r)\neq 0$.
Thus $rq\neq 0$,
contradicting $r\in I^{**}=0_R:I^{*}$.

Therefore $I^{**}\subseteq I$,
and hence $I=I^{**}$.

\smallskip
\noindent
(2)$\Rightarrow$(3)
Immediate by taking $I=\fkm$
for $\fkm\in w\text{\rm-}\Max(R)$.

\smallskip
\noindent
(3)$\Rightarrow$(1)

Fix $\fkm\in w\text{\rm-}\Max(R)$.
If $\fkm^{**}=\fkm$, then necessarily
$\fkm^{*}\neq 0$;
otherwise $\fkm^{*}=0$ would give
\[
\fkm^{**}
=
0_R:\fkm^{*}
=
R,
\]
contradicting $\fkm^{**}=\fkm\neq R$.

Hence
\[
\Hom_R(R/\fkm,E)
\ \cong\
\fkm^{*}
\ \neq\
0.
\]
Since $R/\fkm$ is $\GV$-torsion-free
(because $\fkm$ is a $w$-ideal),
we obtain
\[
\hom_w(R/\fkm,E)
=
\Hom_R(R/\fkm,E)
\neq
0
\qquad
\text{for all }\fkm\in w\text{\rm-}\Max(R).
\]

By Theorem~\ref{thm:w-ishi-31},
this cyclic nonvanishing is equivalent to
$w$-faithful injectivity of $E$.
\end{proof}

\begin{proposition}
\label{prop:domain-matlis-1}
Let $R$ be an integral domain with quotient field $Q$, and let 
$\fkm\in w\text{\rm-}\Max(R)$.  
Assume that $R_\fkm$ is a discrete valuation ring
(equivalently, a rank-one valuation domain with principal maximal ideal).

Then the $R_\fkm$-module $Q/R_\fkm$ is faithfully injective.
In particular,
\[
E^{(w)}\ :=\ \prod_{\fkm\in w\text{\rm-}\Max(R)}\big(Q/R_\fkm\big)
\]
is a $w$-universal injective $R$-module.

Consequently, for every nonzero ideal $I\lhd R$ and every
$\fkm\in w\text{\rm-}\Max(R)$, the Matlis evaluation map
\[
\theta:\ R_\fkm/I_\fkm\ \longrightarrow\ 
\Hom_{R_\fkm}\!\big(
\Hom_{R_\fkm}(R_\fkm/I_\fkm,\,Q/R_\fkm),\,Q/R_\fkm
\big),
\qquad
a\longmapsto(\varphi\mapsto \varphi(a)),
\]
is injective.
\end{proposition}
\begin{proof}
Fix $\fkm\in w\text{\rm-}\Max(R)$ and consider $Q/R_\fkm$
as an $R_\fkm$-module.

Since $R_\fkm$ is a discrete valuation ring, it is a principal ideal domain.
Over a principal ideal domain, the injective hull of the residue field
is isomorphic to $Q/R_\fkm$, and $Q/R_\fkm$ is divisible.
Hence, by classical Matlis theory, $Q/R_\fkm$ is an injective cogenerator
over $R_\fkm$.

In particular, for every nonzero ideal $J\lhd R_\fkm$ one has
\[
\Hom_{R_\fkm}(R_\fkm/J,\,Q/R_\fkm)\ \neq\ 0,
\]
which shows that $Q/R_\fkm$ is faithfully injective as an
$R_\fkm$-module.

Now define
\[
E^{(w)}\ :=\ \prod_{\fkm\in w\text{\rm-}\Max(R)}(Q/R_\fkm).
\]

Each $Q/R_\fkm$ is divisible over the domain $R$,
hence injective as an $R$-module.
Therefore $E^{(w)}$, being a product of injective $R$-modules,
is injective over $R$.

For each $\fkm\in w\text{\rm-}\Max(R)$ we have
\[
\Hom_R(R/\fkm,\,E^{(w)})
\ \cong\
\Hom_{R_\fkm}(R_\fkm/\fkm R_\fkm,\,Q/R_\fkm)
\ \neq\ 0.
\]
Thus every $w$-maximal ideal lies in the support of $E^{(w)}$.
By Corollary~\ref{cor:support-test-w-UI},
$E^{(w)}$ is $w$-universal injective.

Finally, let $0\neq I\lhd R$ and fix $\fkm\in w\text{\rm-}\Max(R)$.
Working over the local ring $R_\fkm$, the faithful injectivity of
$Q/R_\fkm$ implies that for every nonzero element
$a\in R_\fkm/I_\fkm$ there exists
\[
\varphi\in \Hom_{R_\fkm}(R_\fkm/I_\fkm,\,Q/R_\fkm)
\quad\text{with}\quad \varphi(a)\neq 0.
\]
Hence $\theta(a)\neq 0$, and therefore the Matlis evaluation map
$\theta$ is injective.
\end{proof}

\paragraph{{\bf  Funding}}   H. Kim was supported by the Basic Science Research Program through the National Research Foundation of Korea (NRF), funded by the Ministry of Education (2021R1I1A3047469).

\end{document}